\documentclass[11pt, notitlepage]{amsart}
\usepackage[utf8]{inputenc}

\usepackage{datetime}
\usepackage{graphicx}
\usepackage[nointegrals]{wasysym}
\usepackage{mathtools}
\usepackage{tikz}
\usepackage{tikz-cd}
\usepackage{enumerate}
\usepackage{amsmath}
\usepackage{amsfonts, amssymb, amsthm}
\usepackage{bbm}
\usepackage[normalem]{ulem}
\usepackage{calc}
\usepackage{hyperref}
\hypersetup{
  colorlinks   = true, 
  urlcolor     = blue, 
  linkcolor    = blue, 
  citecolor   = blue 
}

\newtheorem{thm}{Theorem}[section]

\newtheorem{lem}[thm]{Lemma}

\newtheorem{prop}[thm]{Proposition}

\newtheorem{manualtheoreminner}{Theorem}
\newenvironment{manualtheorem}[1]{%
  \renewcommand\themanualtheoreminner{#1}%
  \manualtheoreminner
}{\endmanualtheoreminner}

\theoremstyle{definition}
\newtheorem{question}[thm]{Question}

\theoremstyle{definition}
\newtheorem{example}[thm]{Example}

\theoremstyle{definition}

\theoremstyle{remark}

\theoremstyle{definition}

\usepackage[top=2cm, bottom=2cm, inner=2cm, outer=2cm]{geometry}

\usepackage{enumerate}
\usepackage{xcolor}

\usepackage{import}
\usepackage{xifthen}
\usepackage{pdfpages}
\usepackage{transparent}

\makeatletter
\def\paragraph{\@startsection{paragraph}{4}%
  \z@\z@{-\fontdimen2\font}%
  {\normalfont\bfseries}}
\makeatother

\makeatletter
\renewcommand{\paragraph}{%
  \@startsection{paragraph}{4}%
  {\z@}{2ex \@plus 1ex \@minus .2ex}{-1em}%
  {\normalfont\normalsize\bfseries}%
}
\makeatother

\graphicspath{{images/}}

\newcommand{\RR}{\mathbb{R}}
\newcommand{\ZZ}{\mathbb{Z}}
\newcommand{\wt}[1]{\widetilde{#1}}
\newcommand{\no}{\mathcal{N}}
\newcommand{\os}{\mathcal{S}}
\DeclareMathOperator{\Mod}{Mod}
\DeclareMathOperator{\Spec}{Spec}

\title[Pseudo-Anosov homeomorphisms of non-orientable surfaces]{Pseudo-Anosov homeomorphisms of punctured non-orientable surfaces with small stretch factor}

\author{Sayantan Khan}
\address{Department of Mathematics, University of Michigan, 530 Church Street, Ann Arbor, MI 48109}
\email{saykhan@umich.edu}
\urladdr{\url{https://www-personal.umich.edu/~saykhan/}}

\author{Caleb Partin}
\address{Department of Mathematics, Georgia Institute of Technology, 686 Cherry Street, Atlanta, GA 30332}
\email{ctpartin@gmail.com}

\author{Rebecca R. Winarski}
\address{ Department of Mathematics and Computer Science, College of the Holy Cross, 1 College Street, Worcester, MA 01610}
\email{rwinarski@holycross.edu}
\urladdr{\url{https://sites.google.com/site/rebeccawinarski/}}



\begin{document}
\maketitle

\begin{abstract}
  We prove that in the non-orientable setting, the minimal stretch factor of a pseudo-Anosov homeomorphism of a surface of genus $g$ with a fixed number of punctures is asymptotically on the order of $\frac{1}{g}$.
  Our result adapts the work of Yazdi to non-orientable surfaces.
  We include the details of Thurston's theory of fibered faces for non-orientable 3-manifolds.
\end{abstract}

\section{Introduction}
\label{sec:introduction}

Let $S_{g,n}$ be a surface of genus $g$ with $n$ punctures.  The mapping class group of $S_{g,n}$ consists of homotopy classes of orientation preserving homeomorphisms of $S_{g,n}$.  The Nielsen--Thurston classification of mapping classes (elements of the mapping class group) says that each mapping class is periodic, preserves some multicurve, or has a representative that is pseudo-Anosov.  For each pseudo-Anosov homeomorphism $\varphi:S_{g,n}\rightarrow S_{g,n}$, the stretch factor $\lambda(\varphi)$ is an algebraic integer that describes the amount by which $\varphi$ changes the length of curves.  Arnoux--Yoccoz \cite{AY} and Ivanov \cite{ivanov} prove that the set
$$\Spec(S_{g,n})=\{\log(\lambda(\varphi)) \mid \varphi \text{ is a pseudo-Anosov homeomorphism of }S_{g,n}\}$$ is a closed discrete subset of $(0,\infty)$. The minimum of $\Spec(S_{g,n})$:
$$\ell_{g,n}=\min\{\log(\lambda(\varphi)) \mid \varphi \text{ is a pseudo-Anosov homeomorphism of }S_{g,n}\}$$ quantitatively describes both the dynamics of the mapping class group of $S_{g,n}$ and the geometry of the moduli space of $S_{g,n}$.

Penner \cite{penner1991bounds} showed that for orientable surfaces, $$\ell_{g,0}\asymp \frac{1}{g}.$$
  Penner conjectured that $\ell_{g,n}$ will have the same asymptotic behavior for  $n\geq0$ punctures.  Following Penner, substantial attention has been given to finding bounds for $\ell_{g,n}$ \cite{AD,bauer,hironaka,HK,HK20,KT,Loving,minakawa}, calculating $\ell_{g,n}$ for specific values of $(g,n)$ \cite{CH,HS,LT,SKL}, and finding asymptotic behavior of $\ell_{g,n}$ for {\it orientable} surfaces with $n\geq 0$ \cite{KT,tsai2009asymptotic,valdivia,yazdi2018pseudo}.  We adapt a result of Yazdi \cite{yazdi2018pseudo} to non-orientable surfaces.

\begin{thm}\label{thm:stretch1}
  Let $\no_{g,n}$ be a non-orientable surface of genus $g$ with $n$ punctures, and let $\ell_{g,n}'$ be the logarithm of
  the minimum stretch factor of the pseudo-Anosov mapping classes acting on $\no_{g,n}$.
  Then for any fixed $n \in \mathbb{N}$, there is a positive constant $B'_1 = B'_1(n)$ and $B'_2 = B'_2(n)$ such
  that for any $g \geq 3$,
  the quantity $\ell_{g,n}'$ satisfies the following inequalities:
  \begin{align*}
    \frac{B'_1}{g} \leq \ell'_{g,n} \leq \frac{B'_2}{g}.
  \end{align*}
\end{thm}

\paragraph{Pseudo-Anosov homeomorphisms} Let $S$ be a (possibly non-orientable) surface of finite type.  A homeomorphism $\varphi:S\rightarrow S$ is said to be {\it pseudo-Anosov} if there exist a pair of transverse measured singular foliations $\mathcal{F}_s$ and $\mathcal{F}_u$ and a real number $\lambda$ such that $$\varphi(\mathcal{F}_s)=\lambda^{-1} \cdot \mathcal{F}_s\text{ and } \varphi(\mathcal{F}_u)=\lambda \cdot \mathcal{F}_u.$$  The {\it stretch factor} of $\varphi$ is the algebraic integer $\lambda=\lambda(\varphi)$.

Endow $S$ with a Riemannian metric.  The stretch factor $\lambda(\varphi)$ measures the growth rate of the length of geodesic representatives of a simple closed curve $S$ under iteration of $\varphi$ \cite[Proposition 9.21]{FLP}.  Moreover, $\log(\lambda(\varphi))$ is the minimal topological entropy of any homeomorphism of $S$ that is isotopic to $\varphi$ \cite[Expos\'e 10]{FLP}.

\paragraph{Geometry of moduli space}
Let $\mathcal{T}_{g,n}$ denote the Teichm\"uller space of $S_{g,n}$, that is: the space of isotopy classes of hyperbolic metrics on $S_{g,n}$.
When endowed with the Teichm\"uller metric, the mapping class group of $S_{g,n}$ acts properly discontinuously on $\mathcal{T}_{g,n}$ by isometries.  The quotient of this action is the {\it moduli space} of $S_{g,n}$. The set
$\Spec(S_{g,n})$ is the length spectrum of geodesics in the moduli space of $S_{g,n}$.  Therefore the quantity $\ell_{g,n}$ is the length of the shortest geodesic in the moduli space of $S_{g,n}$.

\paragraph{Explicit bounds} In his foundational work, Penner found $\frac{\log 2}{12g-12+4n}$ to be a lower bound for $\ell_{g,n}$ for orientable surfaces \cite{penner1991bounds}.  He also determined $\frac{\log 11}{g}$ to be an upper bound for $\ell_{g,0}$.  Penner's work proves that $\ell_{g,0}\asymp \frac{1}{g}$.  McMullen  \cite{mcmullen2000polynomial} later asked:
\begin{question}[McMullen]
Does $\displaystyle\lim_{g\rightarrow\infty}g\cdot \ell_{g,0}$ exist, and if so, what does it converge to?
\end{question}
To this end, Bauer \cite{bauer} strengthened the upper bound for $g\cdot \ell_{g,0}$ to $\log 6$, and Minakawa \cite{minakawa} and Hironaka--Kin \cite{HK} further sharpened the upper bounds for $g\cdot \ell_{g,0}$ and $g\cdot \ell_{0,2g+1}$ to $\log(2+\sqrt{3})$.  Later Aaber--Dunfield \cite{AD}, Hironaka \cite{hironaka}, and Kin-Takasawa \cite{KTbounds} determined that $\log\left(\frac{3+\sqrt{5}}{2}\right)$ is an upper bound for $g\cdot \ell_{g,0}$ and conjectured it is the supremum of $g\cdot \ell_{g,0}$.

\paragraph{Asymptotic behavior of punctured surfaces}
Tsai initiated the study of asymptotic behavior of $\ell_{g,n}$ along lines in the $(g,n)$-plane \cite{tsai2009asymptotic}.  In particular, Tsai determined that for orientable surfaces of fixed genus $g\geq 2$, the asymptotic behavior in $n$ is:
$$\ell_{g,n}\asymp \frac{\log n}{n}.$$
Further, she showed that $\ell_{0,n}\asymp \frac{1}{n}.$
Later, Yazdi \cite{yazdi2018pseudo} determined that for an orientable surface with a fixed number of punctures $n\geq 0$, the asymptotic behavior in $g$ is:
$$\ell_{g,n}\asymp \frac{1}{g},$$
confirming the conjecture of Penner.

\paragraph{Non-orientable surfaces}
Let $\no_{g,n}$ be a non-orientable surface of genus $g$ with $n$ punctures.  As above, let $\ell'_{g,n}$ denote the minimum stretch factor of pseudo-Anosov homeomorphisms of $\no_{g,n}$.  For any $n\geq 0$ and $g\geq 1$, $\ell_{g-1,2n}$ is a lower bound for $\ell'_{g,n}$, which can be seen by passing to the orientation double cover of $\no_{g,n}$ (note that the definition of genus is different for orientable and non-orientable surfaces).  Because the upper bounds for $\ell_{g,n}$ are constructed by example, upper bounds for $\ell'_{g,n}$ do not follow immediately from passing to the orientation double cover.  Recently Liechti--Strenner determined $\ell'_{g,0}$ for $g\in\{4,5,6,7,8,10,12,14,16,18,20\}$ \cite{LS}.  Our work captures the asymptotic behavior for the punctured case.


\paragraph{Techniques} To prove Theorem \ref{thm:stretch1}, we adapt the strategy of Yazdi \cite{yazdi2018pseudo} to non-orientable surfaces with punctures.  The lower bound of $\ell'_{g,n}$ is found by lifting to the orientation double cover of $\no_{g,n}$.  The upper bound (as in all prior work) is constructive.  Fix $n\geq 0$: the desired number of punctures.  Yazdi's construction is as follows.  For a sequence of genera $g_{n,k}$ (where $k$ goes from $3$ to $\infty$, and $g_{n,k} = (14k-2)n + 2$), use the Penner construction \cite{penner1988construction} to obtain a homeomorphism $f_{n,k}$ of $S_{g_{n,k},n}$ that is pseudo-Anosov and has low stretch factor.  In order to find pseudo-Anosov homeomorphisms of $S_{g,n}$ with small stretch factor for all $g$ (not just those in the sequence above), construct a mapping torus for each $f_{n,k}$.  To do this Yazdi's appeals to a technique involving the use of Thurston's theory of fibered faces.

\paragraph{Thurston norm for non-orientable 3-manifolds} In Thurston's development of what is now called the Thurston norm for 3-manifolds \cite{thurston1986norm}, his definitions and theorems required that all surfaces were orientable.  Thurston said that the theorems should still be true for non-orientable surfaces, but there are some subtleties that have not been addressed elsewhere in the literature.  In this paper, we write the details of Thurston's theory of fibered faces to non-orientable 3-manifolds.  In particular, for orientable 3-manifolds, the Thurston norm is a norm on the second homology of a 3-manifold that measures the minimum complexity of an embedded (orientable) surface; it will need to be adjusted in non-orientable 3-manifolds. Specifically, the Thurston norm does not recognize embedded non-orientable surfaces in the second homology of a non-orientable 3-manifold.  To address this limitation, we instead calculate the Thurston norm on the first cohomology of a non-orientable manifold.  We develop a (weak) version of Poincar\'e duality in Theorem \ref{thm:strong-duality} that suffices to define a Thurston norm on $H^1(M;R)$ for a non-orientable 3-manifold $M$.

\paragraph{Fibered faces} A special case of Thurston's hyperbolization theorem says that the monodromy of any fibration of a hyperbolic 3-manifold over $S^1$ is a pseudo-Anosov homeomorphism.  Therefore by finding other fibrations of the same 3-manifold, one obtains additional pseudo-Anosov homeomorphism.  Work of Fried \cite{fried1982flow,fried1983transitive}, Matsumoto \cite{matsumoto1987topological}, and Agol--Leininger--Margalit \cite{agol6983pseudo} can be used to bound the stretch factors of certain pseudo-Anosov homeomorphisms obtained in this way.

\paragraph{Outline} In Section \ref{sec:thur-norm-non-orientable} we state Thurston's theory of fibered faces and adapt it to the non-orientable setting.  In Section \ref{sec:mapping-classes-with} we show how Thurston's theory of fibered faces can be used to construct pseudo-Anosov homeomorphisms of low stretch factor for non-orientable surfaces.  Specifically, we state and prove the Nielsen--Thurston classification for non-orientable surfaces.  Then we adapt the results of Fried \cite{fried1982flow,fried1983transitive}, Matsumoto \cite{matsumoto1987topological}, and Agol--Leininger--Margalit \cite{agol6983pseudo} used to construct pseudo-Anosov homeomorphisms with low stretch factor of orientable surfaces to the non-orientable setting.  In Section \ref{sec:application}, we prove Theorem \ref{thm:stretch1}, following the strategy of Yazdi.

\paragraph{Acknowledgements}
This work is the result of an REU at the University of Michigan in summer 2020.  We are grateful to Alex Wright for organizing the REU and to Livio Liechti for suggesting the project. We are thankful to Dan Margalit for his comments and to anonymous referees for their comments.  We are also grateful to the other co-organizers, mentors, and participants in the REU: Paul Apisa, Chaya Norton, Christopher Zhang, Bradley Zykoski, Anne Larsen, and Rafael Saavedra.  The REU was partially funded by NSF Grant DMS 185615.  The third author acknowledges support of the NSF through grant 2002951.

\section{Thurston norm for non-orientable 3-manifolds}
\label{sec:thur-norm-non-orientable}
Thurston defined a norm on $H_2(M;\RR)$ where $M$ is an orientable 3-manifold \cite{thurston1986norm}, and this norm is now called the Thurston norm.  In his manuscript, Thurston wrote ``Most of this paper works also for non-orientable manifolds but for simplicity we only deal with the orientable case.''  However, the details are not explained in Thurston's work or in subsequent literature.  Therefore the goal of this section is to write the details of the Thurston norm for non-orientable 3-manifolds.
We recall the Thurston norm for orientable manifolds in Section \ref{sec:backgr-thurst-norm}.
In Section \ref{sec:thurst-norm-cohom} we describe the challenge of defining the Thurston norm on $H_2(M;\RR)$ if $M$ is non-orientable and present the solution of defining the Thurston norm instead on $H^1(M;\RR)$.
However, Poincar\'e duality does not hold for non-orientable manifolds.
We therefore define a condition -- {\it relative orientability} -- on a pair consisting of a manifold and an embedded surface.
A surface that is relatively orientable in a non-orientable 3-manifold $M$ will have a corresponding cohomology class in $H^1(M;\ZZ)$, giving a version of Poincar\'e duality for non-orientable 3-manifolds as stated in Theorem~\ref{thm:strong-duality}.
Finally in Section \ref{sec:oriented-sums}, we define the oriented sum for relatively oriented embedded surfaces in non-orientable manifolds.

\subsection{Thurston norm and mapping tori}
\label{sec:backgr-thurst-norm}
In this section we recall the Thurston norm for orientable surfaces and how it detects when a 3-manifold fibers over a circle.

\paragraph{Mapping tori} Let $S$ be a surface, and let $\varphi: S \to S$ be a homeomorphism.  A {\it mapping torus} of $S$ by $\varphi$ is the $3$-manifold $M_\varphi$ given by the identification:
\begin{align*}
  M_\varphi \coloneqq \frac{S \times [0,1]}{(x,1) \sim (\varphi(x), 0)}.
\end{align*}

A mapping torus is \emph{fibration over $S^1$}, denoted $S\rightarrow M_\varphi\rightarrow S^1$.
A fibration defines a flow on $M$, called the \emph{suspension flow}, where for any $x_0\in S$ and $t_0\in S^1$ the pair $(x_0,t_0)$ is sent to $(x_0,t_0+t)$.
The fiber of a fibration is the preimage of any point $\theta \in S^1$ under the projection map from $M_{\varphi} \to S^1$.
If we do not specify $\theta$, the fiber as a subset of $M_\varphi$ is only well defined up to isotopy.   The homology class of the fiber in $H_2(M_{\varphi}; \RR)$ is well-defined.

A natural inverse question is to determine when a 3-manifold fibers over a circle, and the possible fibers.  To this end, Thurston established a correspondence between second homology of 3-manifolds and surfaces embedded in 3-manifolds.

\paragraph{Complexity of an embedded surface} Let $M$ be a compact orientable closed $3$-manifold.
Let $S$ be a connected surface embedded in $M$.  The complexity of $S$ is $\chi_-(S) = \max\left\{-\chi(S),0\right\}$.
If the surface $S$ has multiple components $S_1, \ldots, S_m$ then $\chi_-(S) = \displaystyle\sum_{i=1}^m\chi_-(S_i)$.
The elements in $H_2(M ; \ZZ)$ can be represented by embedded surfaces inside of $M$ \cite[Lemma 1]{thurston1986norm}.

\paragraph{Thurston norm} Let $a$ be a homology class in $H_2(M; \ZZ)$.  Define the integer valued norm $x:H_2(M;\ZZ)\rightarrow \ZZ$ as the following:
\begin{align*}
  x(a) = \min\{\chi_-(S) \mid [S] = a \text{ and $S$ is compact, properly embedded and oriented}\}.
\end{align*}

We then linearly extend $x$ to $H_2(M;\mathbb{Q})$.  The \emph{Thurston norm} is the unique continuous $\RR$ valued function that is an extension of $x$ to $H_2(M;\RR)$.
The unit ball for the Thurston norm is a convex polyhedron in $H_2(M;\RR)$.

The following remarkable theorem of Thurston \cite{thurston1986norm} determines all possible fibrations of an oriented 3-manifold over the circle.
We use the restatement of Yazdi \cite{yazdi2018pseudo}.
\begin{thm}[Thurston]
  \label{thm:Thur1}
  Let $M$ be an orientable 3-manifold.  Let $\mathcal{F}$ be the set of homology classes in $H_2(M; \RR)$ that are representable by fibers of fibrations of $M$ over the circle.
\begin{enumerate}[(i)]
\item Elements of $\mathcal{F}$ are in one-to-one correspondence with (non-zero) lattice points inside some union of cones over open faces of the unit ball in the Thurston norm.
\item If a surface $F$ is transverse to the suspension flow associated to some fibration of
  $M \xrightarrow[]{} S^1$ then $[F]$ lies in the closure of the corresponding cone in $H_2(M;\RR)$.
\end{enumerate}
\end{thm}
The class $[F]$ has orientation such that the positive flow direction is pointing outwards relative to the surface.
An open face of the unit ball is said to be a \emph{fibered face} if the cone over the face contains the fibers of a fibration.

The goal for the rest of this section is to prove a version of Theorem \ref{thm:Thur1} for compact non-orientable $3$-manifolds.
Most of the work in the proof will involve reducing the version for non-orientable $3$-manifolds to the orientable version by passing to the double cover.

\subsection{Thurston norm on cohomology of non-orientable mapping tori}
\label{sec:thurst-norm-cohom}

Let $\no$ be a compact non-orientable surface.
A na\"ive first attempt at defining the Thurston norm would be to define it on the $H_2(\no;\RR)$, like in the orientable case.
However, if the norm is defined on $H_2(\no;\RR)$, the non-orientable version of Theorem \ref{thm:Thur1} will not be true.
Let $\varphi: \no \to \no$ be a homeomorphism and let $N_{\varphi}$ be associated mapping torus.
Clearly, $N_{\varphi}$ fibers over $S^1$, and $\no$ is the fiber of this fibration.  However, the homology class associated to $\no$ is the zero homology class, since the top-dimensional homology of non-orientable compact surfaces is $0$-dimensional.

Our workaround for this problem will be to define a norm on the first cohomology $H^1(N_{\varphi})$ rather than the second homology $H_2(N_{\varphi})$.
By Poincar\'e duality they are isomorphic for orientable 3-manifolds, but that is not true for non-orientable $3$-manifolds.

\paragraph{Poincar\'e Duality}
To see why Poincar\'e duality fails for non-orientable 3-manifolds, we will work through the construction of the isomorphism between first cohomology and second homology for orientable 3-manifolds.  Let $M$ be a 3-manifold.
To define the Poincar\'e dual of $H^1(M;\ZZ)$, we first define a homotopy class of maps $M\rightarrow S^1$.  Then we construct an element of $H_2(M;\ZZ)$.
Let $\alpha$ be a 1-form on $M$ and $[\alpha]$ its class in $H^1(M; \ZZ)$.  Fix a basepoint $y_0\in M$.
The associated map $f_{\alpha}:M\rightarrow S^1$ is given by:
\begin{align}\label{form:map}
  f_{\alpha}(y) \coloneqq  \int_{y_0}^y \alpha \mod \ZZ.
\end{align}
The choice of basepoint does not affect the homotopy class of $f_\alpha$
(see Section 5.1 of \cite{calegari2007foliations} for the details).

Now let $\theta \in S^1$ be a regular value so that $S = f^{-1}_{\alpha}(\theta)$ is a surface.
To construct the associated element of $H_2(M;\ZZ)$, we choose an orientation on $S$ by assigning positive values of $\alpha$ to the outward pointing normal vectors on $S$.
Then $S$ inherits an orientation from the orientation on $M$, and we have defined a
fundamental class $[S]\in H_2(M;\ZZ)$.
We claim that $[S]$ is the Poincar\'e dual to $\alpha$.
\begin{lem}
  \label{lem:poincare-duality}
  Let $\theta$ and $\theta'$ be two regular values of the function $f_{\alpha}$ and let $S=f_\alpha^{-1}(\theta)$ and $S'=f_\alpha^{-1}(\theta')$.
  Then for any closed $2$-form $\omega$ on $M$, the following identities hold:
  \begin{enumerate}[(i)]
  \item $\displaystyle
    \int_{S} \omega = \displaystyle\int_{S'} \omega$ and
 \item $\displaystyle
    \int_S \omega = \displaystyle\int_M \alpha \wedge \omega.$
  \end{enumerate}
  In particular, the homology class of $S$ is Poincar\'e dual to $\alpha$.
\end{lem}
\begin{proof}
  To see (i), observe that $S$ and $S'$ are homologous, ie $f^{-1}_{\alpha}([\theta, \theta'])$ is a singular $3$-chain that has $S$ and $S'$ as boundaries.
  From Stokes' theorem, we have the following:
  \begin{align*}
    \int_{S - S'} \omega &= \int_{f_{\alpha}^{-1}([\theta, \theta'])} d\omega
                         = 0.
  \end{align*}

  To prove (ii), observe that because $\alpha$ is the pullback of $d\xi$ along the map $f_{\alpha}$ we can write the right hand side:
  \begin{align*}
    \int_M \alpha \wedge \omega &= \int_{S^1} \left(   \int_{f_{\alpha}^{-1}(\xi)} \omega \right) d\xi.
  \end{align*}
  By Sard's theorem, almost every $\xi \in [0,1]$ is a regular value.  Therefore the right-hand side is well-defined.
  By (i), the inner integral is a constant function, as we vary over the $\xi$ which are regular values of $f_{\alpha}$.
  Then the integral of $d\xi$ over $S^1$ is $1$, giving us the identity we want:
  \begin{align*}
    \int_M \alpha \wedge \omega = \int_S \omega.
  \end{align*}
\end{proof}
What we have here is an explicit formula for the Poincar\'e duality map from $H^1(M; \RR)$ to $H_2(M; \RR)$.
For orientable $3$-manifolds, this is an isomorphism.
\begin{thm}[Poincar\'e duality for orientable $3$-manifolds]
  \label{thm:orientable-poincare-duality}
  Let $M$ be an orientable $3$-manifold, and let $S$ be an orientable embedded surface. Then there exists a $1$-form
  $\alpha$ and a regular value $\theta \in S^1$ such that $S$ and $f_{\alpha}^{-1}(\theta)$ are homologous surfaces.
\end{thm}

Let $N$ be a non-orientable 3-manifold.  The map above from $H^1(N;\RR)$ to $H_2(N;\RR)$ is still well-defined.
However the map from $H^1(N; \ZZ)$ to $H_2(N; \ZZ)$ has a nontrivial kernel.  For example, when $N_{\varphi}$ is the mapping torus of a non-orientable surface $\no$, as above, the fiber is trivial in $H_2(N;\ZZ)$.

\paragraph{Non-orientable manifolds}
Let $N$ be a non-orientable $3$-manifold.  Let $\wt{N}$ and the covering map $p:\wt{N}\rightarrow N$ be the orientation double covering space of $N$.
Let $\iota$ be the orientation reversing deck transformation
of $\wt{N}$.
Let $N=N_\varphi$ be the mapping torus of the non-orientable surface $\no$ by a homeomorphism $\varphi: \no \to \no$. Then $\wt{N}$ is the mapping torus of $(\os, \wt{\varphi})$, where $\os$ is the orientation double cover of $\no$, and $\wt{\varphi}$ is the orientation preserving lift of $\varphi$.

\paragraph{Defining the Thurston norm on cohomology} In order to define the Thurston norm on $H^1(N;\ZZ)$, we first need to relate $H^1(N; \RR)$ and $H^1(\wt{N}; \RR)$.
We do so by pulling back $H^1(N;\RR)$ to $H^1(\wt{N};\RR)$ via $p$. We also state the following lemma without proof (the proof is elementary).

\begin{lem}
  \label{lem:injective}
  The pullback $p^{\ast}:H^1(N;\RR)\rightarrow H^1(\wt{N};\RR)$ maps $H^1(N; \RR)$ bijectively to the $\iota^{\ast}$-invariant subspace of   $H^1(\wt{N}; \RR)$.
\end{lem}

Next we use Lemma \ref{lem:injective} to define the Thurston norm on $H^1(N; \RR)$.

\paragraph{Thurston norm for non-orientable $3$-manifolds}
  Let $\alpha \in H^1(N;\RR)$ and let $\wt{x}$ be the Thurston norm on $H^1(\wt{N};\RR)\cong H_2(\wt{N};\RR)$.
  The \emph{Thurston norm on $H^1(N; \RR)$}, is the norm $x: H^1(N;\RR) \rightarrow \RR$ defined:
  \begin{align*}
    x(\alpha) \coloneqq \wt{x}(p^{\ast}\alpha).
  \end{align*}

Note that defining the Thurston norm on $H^1(N; \RR)$ rather than $H_2(N; \RR)$ is not quite satisfactory.
In particular, fibers of fibrations are embedded surfaces in $N$.  In the orientable case, the embedded surfaces define the Thurston norm.
In Section \ref{sec:weak-inverse-poinc}, we develop a (weak) version of Poincar\'e duality for non-orientable 3-manifolds.

\subsection{Weak inverse to the Poincar\'e duality map}
\label{sec:weak-inverse-poinc}

We state and prove a weak version of Poincar\'e duality for {\it relatively oriented} (non-orientable) surfaces embedded in 3-manifolds as Theorem \ref{thm:strong-duality}.

\paragraph{Relative oriented surfaces}
  Let $M$ be a $3$-manifold, and $S$ an embedded surface in $M$.
  The surface $S$ is said to be \emph{relatively oriented with respect to $M$} if there is a nowhere vanishing vector field on $S$ that is transverse to the tangent plane of $S$.
  Two such vector fields are said to induce the same orientation if they induce the same local orientation after choosing a local frame for $S$.
  A surface $S$ is \emph{relatively oriented} in $M$ if both $S$ and the choice of positive normal vector field are specified.

If $S$ and $M$ are orientable, then $S$ is relatively oriented with respect to $M$.
But even if $M$ is non-orientable, a non-orientable embedded surface $S$ may be relatively oriented in $M$.  In particular, we have the following Lemma.
\begin{lem}
  \label{lem:fibers-relatively-oriented}
  Let $\no$ be the fiber of a fibration $f: N \to S^1$.
  Then $\no$ is relatively oriented in $N$.
\end{lem}
\begin{proof}
Consider a non-zero tangent vector $v$ pointing in the positive direction at a point $\theta \in S^1$.
One can pull back the tangent vector $v$ to a nowhere vanishing vector field over $f^{-1}(\theta) = \no$ because $f$ is a fibration, ie a submersion.
The pulled back vector field defines a relative orientation for $\no$ in $N$.
\end{proof}

\paragraph{Orientable manifolds} Now let $M$ be an orientable 3-manifold, and let $S$ be an orientable embedded surface.  If $S$ is relatively oriented with respect to $M$, then a choice of orientation on $S$ determines an orientation on $M$ and vice versa.
We also need to define the notion of \emph{incompressible surfaces} to state our version of Poincar\'e duality.

\paragraph{Incompressible surfaces}
  Let $S$ be a surface with positive genus embedded in a $3$-manifold $M$.
  The surface $S$ is said to be \emph{incompressible} if there does not exist an embedded disc $D$ in $M$ such that $D$ intersects $S$ transversely and $D \cap S = \partial D$.
  The following result of Thurston demonstrates the link between incompressible surfaces and fibers of fibrations.

\begin{thm}[Theorem 4 of \cite{thurston1986norm}]
  \label{thm:Thur2}
Let $M$ be an oriented 3-manifold that fibers over $S^1$.  Let $S$ be an incompressible surface embedded in $M$.  If $S$ is homologous to a fiber, then $S$ is isotopic to the fiber.
\end{thm}


In the remainder of the section, we will be working with a non-orientable 3-manifold $N$ and an embedded non-orientable surface $\no$.   Let $\wt{N}$ and the covering map $p:\wt{N}\rightarrow N$ be the orientation double covering space of $N$.  Let $\wt{\no}$ be the preimage of $\no$ under $p$.  Let $\iota:\wt{N}\rightarrow\wt{N}$ be the orientation-reversing deck transformation of $p$.  We will initiate $N$ and $\no$ in each result below, but we surpress the initiation of the orientation double cover.

\begin{thm}[Poincar\'e Duality for non-orientable 3-manifolds]
  \label{thm:strong-duality}
  Let $N$ be a compact non-orientable $3$-manifold, and let $\no$ be a relatively oriented incompressible surface embedded in $N$.
  Then there exists $[\alpha] \in H^1(N; \ZZ)$ such that the pullback of $[\alpha]$ to $\wt{N}$ is the Poincar\'e dual of $\wt{\no}$ in $\wt{N}$.
  If $[\alpha]$ has a $1$-form representative $\alpha$ that vanishes nowhere on $N$, then $\no$ is homeomorphic to $f_{\alpha}^{-1}(\theta)$ for all $\theta \in S^1$.
\end{thm}

We will refer to the 1-form $\alpha$ given in Theorem~\ref{thm:strong-duality} as the {\it Poincar\'e dual} of the non-orientable surface $\no$.  Before proving Theorem~\ref{thm:strong-duality}, we need three lemmas.

\begin{lem}
  \label{lem:PD1}
  Let $N$ be a non-orientable 3-manifold.
  Let $\no$ be a relatively oriented embedded surface in $N$, and let $\wt{\no}=p^{-1}(\no)$ in $\wt{N}$.
  Then the Poincar\'e dual to $[\wt{\no}]$ is $\iota^{\ast}$-invariant.
\end{lem}
\begin{proof}
  A positive vector field on $\no$ that is transverse to its tangent plane in $N$ lifts to a relative orientation of $\wt{\no}$ in $\wt{N}$.
  Since $\wt{\no}$ and $\wt{N}$ are orientable, the relative orientation of $\wt{\no}$ defines an orientation of $\wt{\no}$, and thus the homology class $[\wt{\no}]$ in $H_2(\wt{N};\RR)$ is well-defined.

  Next we show that $\iota$ reverses the orientation of $\wt{\no}$.  To do so, we first observe that because $\no$ is relatively oriented in $N$, the outward pointing transverse vector field on ${\no}$ must lift to an outward pointing transverse vector field on $\wt{\no}$.  In particular, for any outward pointing vector $\wt{v}$ on $\no$, the vector $\iota(\wt{v})$ is also outward pointing.

  Lift an outward pointing transverse vector field on $\no$ to an outward pointing transverse vector field $\wt{V}$ on $\wt{\no}$.  Let $(v_1, v_2, v_3)$ be a local frame for some point in $\wt{\no}$ such that
  $v_3$ is in $\wt{V}$.
  Since $\iota$ reverses the orientation of $\wt{N}$ but preserves the direction of $v_3$, $\iota$ must reverse the orientation of the pair $(v_1, v_2)$.
  In particular, that means $\iota$ reverses the orientation of $\wt{\no}$.

  Therefore $[\wt{\no}]$ is in the $(-1)$-eigenspace of the $\iota_{\ast}$ action on $H_2(\wt{N}; \RR)$.
  Let the cohomology class $[\wt{\alpha}]$ be the the Poincar\'e dual to $[\wt{\no}]$.
  Let $\wt{\alpha}$ be a representative $1$-form $\wt{\alpha}$ of $[\wt{\alpha}]$ (that need not be $\iota^{\ast}$-invariant).
  We use the fact that $\iota^2= \mathrm{id}$ in the first and third equalities:
  \begin{align*}
    \int_{\iota_{\ast}\wt{\no}} \omega &= \int_{\wt{\no}} \iota^{\ast}\omega &&\text{(By a change of variables)} \\
                                     &= \int_{\wt{N}} \wt{\alpha} \wedge \iota^{\ast} \omega &&\text{(Poincar\'e duality)} \\
                                     &=\int_{\wt{N}} \iota^{\ast} \left( \iota^{\ast}\wt{\alpha} \wedge \omega \right) \\
    &= \int_{\wt{N}} - \left( \iota^{\ast} \wt{\alpha} \wedge \omega \right) &&\text{($\iota$ is orientation reversing)}.
  \end{align*}
  Because $\iota_{\ast}[\wt{\no}] = -[\wt{\no}]$, we have the following:
  \begin{align*}
    \int_{\iota_{\ast}\wt{\no}} \omega &= - \int_{\wt{\no}} \omega \\
                              &= - \int_{\wt{N}} \wt{\alpha} \wedge \omega.
  \end{align*}
  Since $$\int_{\wt{N}}\wt{\alpha}\wedge\omega=\int_{\wt{N}}\iota^\ast\wt{\alpha}\wedge\omega$$ for all $\omega$, it follows that $\wt{\alpha}$ and $\iota^{\ast}\wt{\alpha}$ differ by an exact form, and therefore the cohomology class $[\wt{\alpha}]$ is
  $\iota^{\ast}$-invariant.
\end{proof}

As above, we will denote the Poincar\'e dual to $[\wt{\no}]$ by $[\wt{\alpha}]$.  The class $[\wt{\alpha}]$ is an $\iota^{\ast}$-invariant element of $H^1(\wt{N}; \ZZ)$, but it is not clear that $[\wt{\alpha}]$ is the pullback of an element of $H^1(N; \ZZ)$ under $p$.
In the next lemma, we show that is indeed the case.
\begin{lem}
  \label{lem:PD2}
Let $N$ be a non-orientable 3-manifold.  Let $[\wt{\alpha}]\in H^1(\wt{N}, Z)$ and let $\wt{S}$ be the Poincar\'e dual of $[\wt{\alpha}]$ in $\wt{N}$.  There exists $[\alpha] \in H^1(N; \ZZ)$ such that $\wt{\alpha} = p^{\ast} \alpha$.
\end{lem}
\begin{proof}
 It will suffice to show that for any simple closed curve $\gamma$ in $N$, the integral of $\wt{\alpha}$ along any path lift of $\gamma$ is an integer.  Let $x_0\in N$ be a base point of $\gamma$.  Note that $\gamma$ has two (path) lifts $\wt{\gamma}_1,\wt{\gamma}_2$ under $p$ in $\wt{N}$, one based at each element of $p^{-1}(x_0)$.  Either $\wt{\gamma}_1,\wt{\gamma}_2$ are both simple closed curves based at the each of the two preimages $p^{-1}(x_0)$ or $\wt{\gamma}_1,\wt{\gamma}_2$ are both arcs between the two points of $p^{-1}(x_0)$.
  If each lift $\wt{\gamma}_1,\wt{\gamma}_2$ of $\gamma$ is a closed curve in $\wt{N}$, the integral $\displaystyle\int_{\wt{\gamma}_i}\wt{\alpha}$ will be an integer since $[\wt{\alpha}] \in H^1(\wt{N}; \ZZ)$ for $i=1,2$.

 If each lift $\wt{\gamma}_1,\wt{\gamma}_2$ of $\gamma$ is an arc between the two preimages of $p^{-1}(x_0)$, we consider the simple closed curve $\wt{\gamma}=\wt{\gamma}_1\cup\wt{\gamma}_2$.  We note that $\iota(\wt{\gamma})=\wt{\gamma}$.
  By Lemma \ref{lem:PD1}, $\wt{\alpha}$ is $\iota^{\ast}$-invariant.  Therefore we have that $ \displaystyle\int_{\wt{\gamma}_1} \wt{\alpha} = \displaystyle\int_{\wt{\gamma}_2} \wt{\alpha}.$  Therefore $$\int_{\wt{\gamma}}\wt{\alpha}=2\int_{\wt{\gamma}_1}\wt{\alpha}.$$ It will suffice to show that $\displaystyle\int_{\wt{\gamma}}\wt{\alpha}$ is an even integer.
  Without loss of generality, we can assume all intersections of the simple closed curve $\wt{\gamma}$ with the surface $\wt{\os}$ are transverse.
  Since $\wt{\alpha}$ is a representative of the Poincar\'e dual to $[\wt{\os}]$, the integral of $\wt{\alpha}$ along $\wt{\gamma}$ is the signed intersection number of $\wt{\gamma}$ with $\wt{\os}$.
  The intersection number must be even, for if $\wt{\gamma}$ and $\wt{\os}$ intersect at a point $y$, then they also intersect at $\iota(y)$. This proves the lemma.
\end{proof}

The last lemma we need is the claim that lifts of incompressible surfaces are incompressible.
\begin{lem}
  \label{lem:lift-of-incompressible}
  Let $N$ be a non-orientable 3-manifold.  If $\no$ is a relatively oriented incompressible surface in $N$, then $\wt{\no}=p^{-1}(\no)$ is incompressible in $\wt{N}$.
\end{lem}
\begin{proof}
  Because $\no$ is incompressible in $N$, the map on fundamental groups induced by inclusion $\no\rightarrow N$ is injective.
  Since $p_\ast:\pi_1(\wt{N})\rightarrow\pi_1(N)$ is injective, the induced map $\pi_1(\wt{\no}) \to \pi_1(\wt{N})$ must also be injective.  An injective induced map on fundamental groups is equivalent to the orientable surface $\wt{\no}$ being incompressible.
\end{proof}

We now have everything we need to finish proving Theorem \ref{thm:strong-duality}.
\begin{proof}[Proof of Theorem \ref{thm:strong-duality}]
  Let $\wt{\no}=p^{-1}(\no).$
  The relative orientation of $\wt{\no}$ determines a homology class $[\wt{\no}]\in H_2(N;\ZZ)$.  Let the $1$-form $\wt{\alpha}$ be the Poincar\'e dual to $[\wt{\no}]$ in $\wt{N}$.
  By Lemma \ref{lem:PD2}, there exists a 1-form $\alpha\in H^1(N;\ZZ)$ such that $\wt{\alpha}=p^\ast\alpha$.

We define the map $f_\alpha:N\rightarrow S^1$ according to equation (\ref{form:map}).  Because $\alpha$ is non-vanishing, $f_{\alpha}$ has full rank everywhere.  Therefore $f_\alpha$ is a fibration.
The map $f_{\alpha} \circ p$ is a lift of $f_{\alpha}$ to $\wt{N}$ under $p$, and is therefore also a fibration.
  By Lemma \ref{lem:lift-of-incompressible}, $\wt{\no}$ is incompressible.
It follows from the orientable version of Poincar\'e duality that $\wt{\no}$ and $p^{-1}(f_{\alpha}^{-1}(\theta))$ are homologous surfaces in $\wt{N}$.    Theorem \ref{thm:Thur2} then tells us $\wt{\no}$ must be isotopic to a fiber of $f_{\alpha} \circ p$.
  The restriction of $p$ to the homeomorphic surfaces $\wt{\no}$ and $p^{-1}(f_{\alpha}^{-1})(\theta)$ determines two equivalent 2-fold covering maps $\wt{\no}\rightarrow \no$ and  $p^{-1}(f^{-1}_{\alpha}(\theta))\rightarrow f^{-1}_\alpha(\theta)$.  Therefore the image surfaces $\no$ and $f_{\alpha}^{-1}(\theta)$ must also be homeomorphic.
\end{proof}

  Note that the above proof does not tell us that $\no$ and $f_{\alpha}^{-1}(\theta)$ are isotopic.  Isotopy of the fibers of $N$ requires the isotopy between $\wt{\no}$ and $p^{-1}(f^{-1}_\alpha(\theta))$ to be $\iota^{\ast}$-equivariant.  However, the theorem is sufficient for our application.

We conclude the section with a non-orientable version of Theorem \ref{thm:Thur1}.
\begin{thm}
  \label{thm:classifying-fibrations}
  Let $N$ be a compact non-orientable $3$-manifold, and let $\mathcal{F}$ be the elements of $H^1(N; \ZZ)$ corresponding to fibrations of $N$ over $S^1$.
  \begin{enumerate}[(i)]
  \item Elements of $\mathcal{F}$ are in one-to-one correspondence with (non-zero) lattice points (ie points of $H^1(N; \ZZ)$) inside some union of cones over open faces of the unit ball in the Thurston norm.
  \item Let $\no$ be relatively oriented surface in $N$ that transverse to the suspension flow associated to some fibration $f: N \to S^1$.  Let $[\alpha]$ be the Poincar\'e dual $[\alpha]$ to $\no$.  Then $[\alpha]$ lies in the closure of the cone in $H^1(N; \RR)$ containing the $1$-form corresponding to $f$.
  \end{enumerate}
\end{thm}
\begin{proof}
For (i), we observe that by Theorem~\ref{thm:Thur1} the fibrations of $\wt{N}$ are in one-to-one correspondence with lattice points inside a union of cones over open faces of the unit ball in $H_2(\wt{N};\RR)$.  Let $\wt{\mathcal{K}}$ be the union of cones in $H_2(\wt{N};\RR)$.
 By Poincar\'e duality, $\wt{\mathcal{K}}$ is in one-to-one correspondence to a union of cones in $H^1(\wt{N};\RR)$, which we will call $\wt{\mathcal{K}}^\ast$.

  Because $H^1(N;\RR)$ is isomorphic to a subspace of $H^1(\wt{N};\RR)$, we can construct a union of cones in $H^1(N; \RR)$ that map to the intersection of $p^\ast(H^1(N; \RR))$ with $\wt{\mathcal{K}}^\ast$.
  Indeed, every lattice point in $\wt{\mathcal{K}}^\ast$ corresponds to a fibration $f:N\to S^1$, since the pullback of $f$ to $H^1(\wt{N}; \ZZ)$ corresponds to a fibration of $\wt{N}$.
  Conversely, every fibration of $f:N\rightarrow S^1$ must correspond to an element of $\wt{\mathcal{K}}^\ast$, since the composition $f\circ p$ is a fibration of $\wt{N}\rightarrow S^1$.

  For (ii), assume that the surface $\no$ is transverse to the suspension flow of a fibration $f:N\rightarrow S^1$. Then $\wt{\no}$ is transverse to the suspension flow $p \circ f:\wt{N}\rightarrow S^1$.  Let $\wt{\alpha}$ be the pullback of $\alpha$ under $p$.  Then $\wt{\alpha}$ is the Poincar\'e dual of $\wt{\no}$.  By \autoref{thm:Thur1}, the 1-form $\wt{\alpha}$ lies in the closure of a component of $\wt{\mathcal{K}}^\ast$ that contains the 1-form corresponding to $f\circ p$.  Let $\wt{K}$ be this component.  Let $K\subset H^1(N;\RR)$ be the preimage of $\wt{K}$ under $p^\ast$.  The cone $K$ contains both $\alpha$ and the 1-form corresponding to $f$, as desired.
\end{proof}

\subsection{Oriented sums}
\label{sec:oriented-sums}

The next step in studying embedded non-orientable surfaces will be to describe \emph{oriented sums}.  Let $M$ be a 3-manifold.
The oriented sum of two embedded surfaces in $M$ is additive in both the Euler characteristic and $H^1(M;\RR)$.
This operation is well-known in the case of orientable $3$-manifolds (along with orientable embedded surfaces), but we will sketch the relevant details.
We then extend the construction to relatively oriented embedded surfaces.

\paragraph{Oriented sum for oriented manifolds}
Let $M$ be an orientable manifold.
Let $S$ and $S'$ be orientable embedded surfaces in $M$.
Assume that $S$ and $S'$ intersect transversally.
Thus $S \cap S'$ is a disjoint union of copies of $S^1$.
For each component of $S\cap S'$, take a tubular neighborhood that has cross section as in \autoref{fig:cross-section}.
\begin{figure}
  \centering
    \def\svgscale{0.2}
\begingroup%
  \makeatletter%
  \providecommand\color[2][]{%
    \errmessage{(Inkscape) Color is used for the text in Inkscape, but the package 'color.sty' is not loaded}%
    \renewcommand\color[2][]{}%
  }%
  \providecommand\transparent[1]{%
    \errmessage{(Inkscape) Transparency is used (non-zero) for the text in Inkscape, but the package 'transparent.sty' is not loaded}%
    \renewcommand\transparent[1]{}%
  }%
  \providecommand\rotatebox[2]{#2}%
  \newcommand*\fsize{\dimexpr\f@size pt\relax}%
  \newcommand*\lineheight[1]{\fontsize{\fsize}{#1\fsize}\selectfont}%
  \ifx\svgwidth\undefined%
    \setlength{\unitlength}{407.39696593bp}%
    \ifx\svgscale\undefined%
      \relax%
    \else%
      \setlength{\unitlength}{\unitlength * \real{\svgscale}}%
    \fi%
  \else%
    \setlength{\unitlength}{\svgwidth}%
  \fi%
  \global\let\svgwidth\undefined%
  \global\let\svgscale\undefined%
  \makeatother%
  \begin{picture}(1,0.80278474)%
    \lineheight{1}%
    \setlength\tabcolsep{0pt}%
    \put(0,0){\includegraphics[width=\unitlength,page=1]{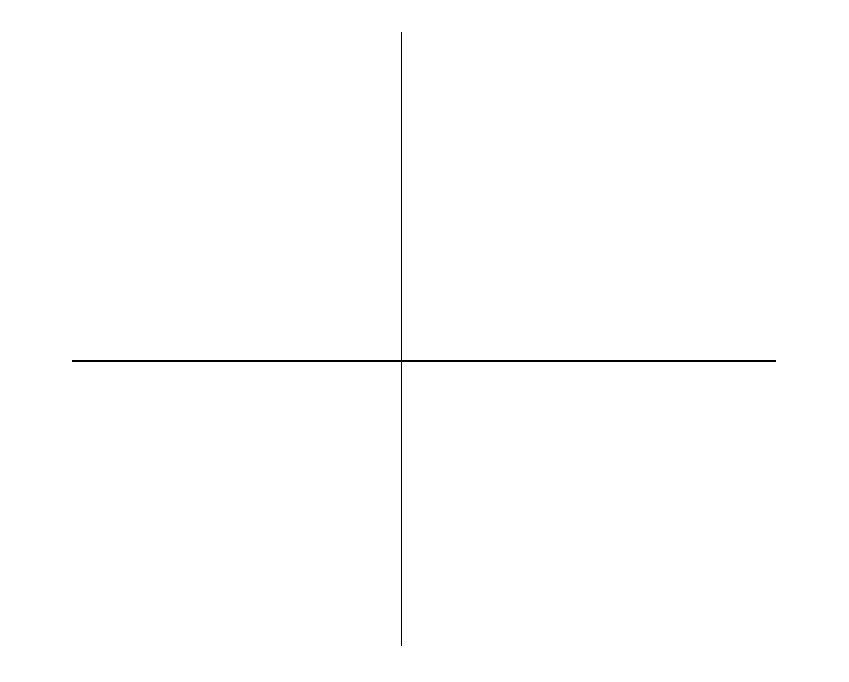}}%
    \put(0.94626714,0.32602052){\color[rgb]{0,0,0}\makebox(0,0)[lt]{\lineheight{1.25}\smash{\begin{tabular}[t]{l}$S$\end{tabular}}}}%
    \put(-0.04244498,0.32652206){\color[rgb]{0,0,0}\makebox(0,0)[lt]{\lineheight{1.25}\smash{\begin{tabular}[t]{l}$S$\end{tabular}}}}%
    \put(0.41724626,-0.07032913){\color[rgb]{0,0,0}\makebox(0,0)[lt]{\lineheight{1.25}\smash{\begin{tabular}[t]{l}$S'$\end{tabular}}}}%
    \put(0.41724626,0.78040563){\color[rgb]{0,0,0}\makebox(0,0)[lt]{\lineheight{1.25}\smash{\begin{tabular}[t]{l}$S'$\end{tabular}}}}%
  \end{picture}%
\endgroup%

  \caption{Cross section of intersection of $S$ and $S'$.}
  \label{fig:cross-section}
\end{figure}
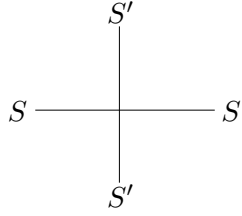

We then perform a surgery on the leaves of $S$ and $S'$ so that the outward pointing normal vector fields match as in \autoref{fig:surgery}.
\begin{figure}[b]
  \centering
    \def\svgscale{0.3}
    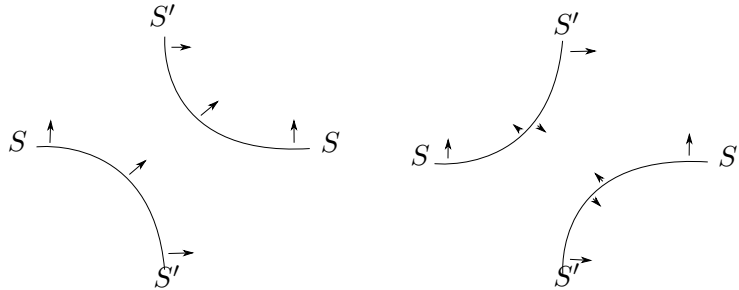

  \caption{On the left, the normal vectors on $S$ and $S'$ are consistent. On the right, they are not.}
  \label{fig:surgery}
\end{figure}

By performing this surgery at all the intersections, we get a new submanifold $S''$ of $M$ (which may have multiple components).
This new submanifold $S''$ is called the {\it oriented sum} of $S$ and $S'$.
The operation of taking oriented sums is additive on Euler characteristic, as well as the homology classes (and thus the cohomology classes of their Poincar\'e duals):
\begin{align*}
  \chi(S'') &= \chi(S) + \chi(S') \\
  [S''] &= [S] + [S']
\end{align*}

\paragraph{Oriented sum for non-orientable manifolds} Let $N$ be a non-orientable 3-manifold and let $\no$ and $\no'$ be embedded surfaces in $N$ that are relatively oriented.
We define the oriented sum on $\no$ and $\no'$ as follows.
As above, let $p:\wt{N}\rightarrow N$ be the orientation double cover and let $\iota$ be the orientation reversing deck transformation of $\wt{N}$.
Let $\wt{\no}=p^{-1}(\no)$ and $\wt{\no}'=p^{-1}(\no')$, which are embedded oriented surfaces in $\wt{N}$.
The oriented sum of $\no$ and $\no'$ is the image under $p$ of the oriented sum of $\wt{\no}$ and $\wt{\no}'$.

To see that the operation is well-defined, we recall that $\iota$ preserves the relative orientation of $\wt{\no}$ and $\wt{\no}'$.  Therefore $\iota$ leaves the outward normal vector fields on $\wt{\no}$ and $\wt{\no}'$ invariant (see the proof of Lemma \ref{lem:PD1}).
Thus a leaf $L$ of $\wt{\no}$ is surgered with a leaf of $L'$ of $\wt{\no}'$ if and only if $\iota(L)$ and $\iota(L')$ are surgered.
Therefore surgery factors through $p$ and $[\no]+[\no']$ is well-defined for non-orientable surfaces.

\begin{example}
  \label{ex:oriented-sum}
   Let $\gamma$ be a component of $\no\cap \no'$ and $\wt{\gamma}_1$ and $\wt{\gamma}_2$ be the path lifts of $\gamma$.
  One possible orientation of $\wt{S}$ and $\wt{S}'$ is given in  \autoref{fig:consistency}.  The outward pointing normal vectors to $\wt{\no}$ and $\wt{\no'}$ determine which leaves are glued together along $\wt{\gamma}_1$ and $\wt{\gamma}_2$.

\begin{figure}
  \centering
    \def\svgscale{0.4}
    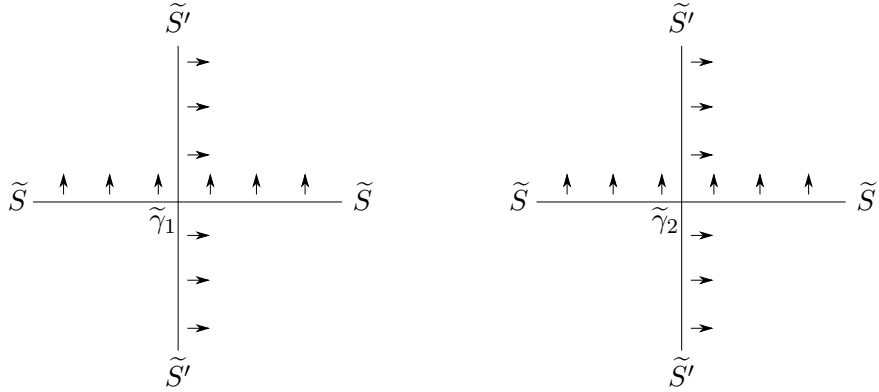

  \caption{Neighborhoods of $\wt{\gamma}_1$ and $\wt{\gamma}_2$, with the outward pointing normal vector field.}
  \label{fig:consistency}
\end{figure}

To preserve the normal vector field, glue the left $\wt{\no}$ leaf to the bottom $\wt{\no'}$ leaf near $\wt{\gamma}_1$ and $\wt{\gamma}_2$.
Since $\iota(\wt{\gamma}_1)=\wt{\gamma_2}$, the outward pointing normal vector fields point the same (relative) directions.
\end{example}

\paragraph{Additivity}
By the consistency of the oriented sum in $N$ and $\wt{N}$, it easily follows that the oriented sum is additive in Euler characteristic, as well as in terms of Poincar\'e dual, since the Poincar\'e dual was also defined by passing to the orientation double cover.

\section{Mapping classes with small stretch factors}
\label{sec:mapping-classes-with}

In this section, we provide a strategy to compute pseudo-Anosov homeomorphisms with small stretch factors.

\subsection{Mapping class groups of non-orientable surfaces}
\label{sec:backgr-mapp-class}
Let $\no$ be a non-orientable surface and let $\wt{\no}$ and the covering map $p:\wt{\no}\rightarrow \no$ be its orientation double covering space.
Every homeomorphism $\varphi: \no \to \no$, has a unique orientation-preserving lift $\wt{\varphi}: \wt{\no} \to \wt{\no}$.

A consequence is that lifting homeomorphisms induces a monomorphism between homeomorphisms of $\no$ and orientation-preserving homeomorphisms of $\wt{\no}$.
Every homotopy of $\no$ lifts to a homotopy of $\wt{\no}$.
Therefore there is an inclusion from the mapping class group of $\no$ to the (orientation-preserving) mapping class group of $\wt{\no}$.
This inclusion also respects the Nielsen-Thurston classification of mapping classes, both qualitatively, and quantitatively, as the following proposition shows.
\begin{prop}
  \label{prop:2}
  Let $\varphi:\no\rightarrow\no$ be a homeomorphism and let $\wt{\varphi}:\wt{\no}\rightarrow\wt{\no}$ be the orientation-preserving lift of $\varphi$.  Then:
  \begin{enumerate}[(i)]
  \item $\varphi$ is periodic if and only if $\wt{\varphi}$ is periodic,
  \item $\varphi$ is reducible if and only if $\wt{\varphi}$ is reducible, and
  \item $\varphi$ is pseudo-Anosov if and only if $\wt{\varphi}$ is pseudo-Anosov.  Moreover if $\varphi$ has stretch factor $\lambda$, then $\wt{\varphi}$ also has stretch factor $\lambda$.
  \end{enumerate}
\end{prop}
\begin{proof}
  The fact that the map from $\Mod(\no)$ to $\Mod(\wt{\no})$ is type-preserving follows from Aramayona--Leininger--Souto \cite[Lemma 10]{aramayona2009injections} (while the statement of the Lemma is for orientable surfaces, the argument, which we will skip, is identical for non-orientable surfaces).

  Suppose now that $\varphi:\no\rightarrow\no$ is a psuedo-Anosov homeomorphism with stretch factor $\lambda$ and stable and unstable foliations $\mathcal{F}_s$ and $\mathcal{F}_u$ respectively.
  Let $\wt{\mathcal{F}_s}$ and $\wt{\mathcal{F}_u}$ denote the lifts of the stable and unstable foliations to the orientation double cover.  Let $\gamma$ be a simple closed curve in $\wt{\no}$.
  We need to show that the following identities hold for all $\gamma$ (see \cite[Expos\'e 5]{FLP} for the definition of intersection number with measured foliations; the fact that these identities suffice follows from \cite[Lemma 9.15]{FLP}):
  \begin{align}
      \label{eq:unstable-foliation}
      i(\gamma, \wt{\varphi}(\wt{\mathcal{F}_u})) &= \lambda \cdot i(\gamma, \wt{\mathcal{F}_u}) \\
      \label{eq:stable-foliation}
      i(\gamma, \wt{\varphi}(\wt{\mathcal{F}_s})) &= \frac{1}{\lambda} \cdot i(\gamma, \wt{\mathcal{F}_s}).
  \end{align}

  To see that (\ref{eq:unstable-foliation}) holds, we partition $\gamma$ into short arcs $\{\gamma_i\}$ such that the restriction of the covering map $p$ to a neighborhood of each arc is a homeomorphism.
  Then we have:
  \begin{align}
  \label{eq:push1}
    i(\gamma_i, \wt{\mathcal{F}_u}) &= i(p(\gamma_i), \mathcal{F}_u) \\
  \label{eq:push2}
    i(\gamma_i, \wt{\varphi}(\wt{\mathcal{F}_u})) &= i(p(\gamma_i), \varphi(\mathcal{F}_u)).
  \end{align}
  Since we know that $\mathcal{F}_u$ is the unstable foliation for $\varphi$ with stretch factor $\lambda$, we can compute the ratio of the right hand side of \eqref{eq:push1} and \eqref{eq:push2}:
  \begin{align}
      \label{eq:ratio}
      i(p(\gamma_i), \varphi(\mathcal{F}_u)) = \lambda \cdot i(p(\gamma_i), \mathcal{F}_u).
  \end{align}
  Combining \eqref{eq:push1}, \eqref{eq:push2}, and \eqref{eq:ratio}, and summing over all $\gamma_i$ gives us that \eqref{eq:unstable-foliation} holds. A similar argument also proves that \eqref{eq:stable-foliation} holds.
\end{proof}

\subsection{Constructing pseudo-Anosov maps using oriented sums}
\label{sec:constr-psuedo-anos}
The goal of this section is to prove that the stretch factor of any pseudo-Anosov homeomorphism provides an asymptotic upper bound for the minimum stretch factor.  We do this in Proposition \ref{prop:asymptotic}.

\begin{prop}\label{prop:asymptotic}
Let $\no_g$ be a non-orientable surface of genus $g$ and let $\varphi:\no_g\rightarrow \no_g$ be a pseudo-Anosov homeomorphism with stretch factor $\lambda$.  Let $N_\varphi$ be the mapping torus of $\no_g$ by $\varphi$.  Let $\no_{g'}$ be a genus $3$ non-orientable relatively orientable surface embedded in $N_\varphi$ that is transverse to the suspension flow associated to $\varphi$.  Then for all $k\in\ZZ^+$, there is a pseudo-Anosov homeomorphism of the oriented sum $\no_{g}+k\no_{g'}$ with stretch factor at most $\lambda$.
\end{prop}

Our strategy for proving Proposition \ref{prop:asymptotic} is to find fibrations of $N_{\varphi}$ over $S^1$ that have fiber $\no_g+k\no_{g'}$.  We then apply a special case of Thurston's hyperbolization theorem, which says that the mapping torus of an orientable surface $S$ by a homeomorphism $\varphi$ is hyperbolic if and only if $\varphi$ is pseudo-Anosov \cite[Theorem 0.1]{thurston_hyp}.  In particular, Thurston's Theorem implies that if $M=M_\varphi$ fibers over $S^1$ in two ways, either both monodromies are pseudo-Anosov or neither monodromy is pseudo-Anosov.  Finally, we adapt theorems of Fried and Matsumoto (Theorem \ref{thm:fm}) and Agol--Leininger--Margalit (Theorem \ref{thm:alm}) to work for mapping tori with non-orientable fibers.

\medskip
We will use the following two facts for orientable surfaces and hyperbolic 3-manifolds that fiber over $S^1$:
\begin{enumerate}
 \item A Thurston norm-minimizing surface $S$ is incompressible \cite[Lemma 5.7]{calegari2007foliations}.
\item  The fiber of any fibration over $S^1$ minimizes the Thurston norm in its homology class \cite[Corollary 2]{thurston1986norm}.
\end{enumerate}

\begin{prop}
  \label{thm:oriented-sum}
  Let $\no'$ be a genus $3$ non-orientable relatively orientable surface embedded in $N$ that is transverse to the suspension flow associated to $\varphi$.
  Let $\alpha$ be the Poincar\'e dual of $\no$ and $\alpha'$ the Poincar\'e dual of $\no'$.
  If the oriented sum of $\no$ and $\no'$ is connected, then $\no + \no'$ is homeomorphic to the fiber of the fibration given by $\alpha + \alpha'$.
\end{prop}
\begin{proof}
  We first need to show that $\no^{\prime}$ is incompressible to consider its Poincar\'e dual: this follows from the fact that the pre-image $\wt{\no^{\prime}}$ in the orientation double cover is a genus $2$ surface, and minimizes the Thurston norm in its homology class.
  If it did not minimize the Thurston norm in the homology class, then the norm minimizing surface in its homology class would have to be a torus or a sphere, but that would contradict the fact the $3$-manifold is the mapping torus of a pseudo-Anosov map. By Calegari \cite[Lemma 5.7]{calegari2007foliations}, we have that $\widetilde{\mathcal{N}'}$ incompressible, and therefore $\no^{\prime}$ is incompressible.

  Let $p:\wt{N}\rightarrow N$ be the orientation double cover of $N$.
  The surface $\no$ minimizes Thurston norm because it is a fiber of $f$.  Similarly, $p^{-1}(\no)$ also minimizes Thurston norm.  Thus the Thurston norm of $\alpha$ is $2\chi_-(\no)$.  Likewise, the Thurston norm of $\alpha'$ is $2\chi_-(\no')$.

By \autoref{thm:classifying-fibrations} (ii), $\alpha'$ lies in the same cone in $H^1(N;\ZZ)$ as $\alpha$.  The Thurston norm $x$ on $H^1(N;\ZZ)$ is a linear function on that cone.
 Since the Thurston norm is also linear on oriented sums of $\no$ and $\no'$, we have:
  \begin{align*}
    x(\alpha + \alpha') &= x(\alpha) + x(\alpha') \\
                        &= 2\chi_-(\no) + 2\chi_-(\no') \\
                        &= 2\chi_-(\no + \no').
  \end{align*}

  Because $2\chi_-(\no+\no')$ achieves the Thurston norm of $\alpha+\alpha'$, the preimage $p^{-1}(\no+\no')$ achieves the Thurston norm of the pullback of $\alpha+\alpha'$ under $p$.  Therefore $p^{-1}(\no+\no')$ is incompressible.  Then $\no+\no'$ is also incompressible.

  By Theorem \ref{thm:classifying-fibrations} (i), we have that $\alpha + \alpha'$ corresponds to some other fibration $f'':N\rightarrow S^1$.
  By Theorem \ref{thm:strong-duality}, the fiber of $f''$ must be homeomorphic to $\no + \no'$.
\end{proof}

In the proof of Proposition \ref{prop:asymptotic}, we will use Proposition \ref{thm:oriented-sum} along with a theorem of Thurston to obtain a pseudo-Anosov homeomorphism $\varphi_k$ of the surface of genus of genus $g+kg'$.  We the use Theorems \ref{thm:fm} and \ref{thm:alm} to obtain a upper bound on the stretch factor of $\varphi_k$.

\begin{thm}[Fried \cite{fried1982flow,fried1983transitive}, Matsumoto \cite{matsumoto1987topological}]
  \label{thm:fm}
  Let $M$ be an orientable hyperbolic $3$-manifold and let $\mathcal{K}$ be the union of cones in $H^1(M; \RR)$ whose lattice points correspond to fibrations over $S^1$.
  There exists a strictly convex function $h: \mathcal{K} \to \RR$ satisfying the following properties:
  \begin{enumerate}[(i)]
  \item For all $c > 0$ and $u \in \mathcal{K}$, $h(cu) =  \frac{1}{c}h(u)$,
  \item For every primitive lattice point $u \in \mathcal{K}$, $h(u) = \log(\lambda)$, where $\lambda$ is the
    stretch factor of the pseudo-Anosov map associated to this lattice point, and
  \item $h(u)$ goes to $\infty$ as $u$ approaches $\partial \mathcal{K}$.
  \end{enumerate}
\end{thm}

\begin{thm}[Agol-Leininger-Margalit]
  \label{thm:alm}
  Let $\mathcal{K}$ be a fibered cone for a mapping torus $M$ and let $\overline{\mathcal{K}}$ be its closure
  in $H^1(M;\RR)$. If $u \in \mathcal{K}$ and $v \in \overline{\mathcal{K}}$, then $h(u+v) < h(u)$.
\end{thm}

\begin{proof}[Proof of Proposition \ref{prop:asymptotic}]
The oriented sum $\mathcal{S}=\no_g+k\no_{g'}$ constructed in Proposition \ref{thm:oriented-sum} is a surface of genus $g+kg'$, and $\mathcal{S}$ is homeomorphic to a fiber of $N_\varphi$ given by $\alpha+k\alpha'$.  Let $\varphi_{k}:\mathcal{S}\rightarrow\mathcal{S}$ be the monodromy of $N_\varphi$ over $\mathcal{S}$.  By Thurston's theorem, $\varphi_k$ is pseudo-Anosov.  We claim that $\varphi_k$ has stretch factor at most $\lambda$.

Let $p:\wt{N}\rightarrow N_\varphi$ be the orientation double cover of $N_\varphi$. Let $h\mid_{N}$ be the restriction of $h$ to the pullback  $p^\ast(H^1(N_\varphi; \RR))$ in $H^1(\wt{N}; \RR)$.
The restriction $h\mid_N$ satisfies all the properties of Theorems \ref{thm:fm} and \ref{thm:alm}.

 Let $\wt{\varphi}$ be the orientation preserving lift of $\varphi$ to $p^{-1}(\no)$.  Since $\wt{\alpha}$ is the pullback of $\alpha$, the map $\wt{\varphi}$ is the pseudo-Anosov homeomorphism associated to $\wt{\alpha}$.  By Proposition \ref{prop:2}, the stretch factor of $\wt{\varphi}$ is $\lambda$.

Let $\mathcal{K}$ be the cone in $H^1(N_\varphi;\RR)$ that contains $\alpha$.  Since $\no_{g'}$ is transverse to the suspension flow in the direction of $\varphi$, we have that $\alpha'$ is in the closure of $\mathcal{K}$ in $H^1(N;\RR)$.  Let $\wt{\alpha}$ be the pullback of $\alpha$ under $p$ and let $\wt{\alpha}'$ be the pullback of $\alpha'$ under $p$.  Then $h\mid_N(\wt{\alpha}+\wt{\alpha}')<h\mid_N(\wt{\alpha})$.  By Theorem \ref{thm:fm}, $h(\wt{\alpha})$ is equal to the stretch factor of the pseudo-Anosov homeomorphism associated to $\wt{\alpha}$.  Therefore we have $h\mid_N(\wt{\alpha}+\wt{\alpha}')<\log(\lambda)$. It follows that the stretch factor of $\varphi_k$ is less than $\lambda$.
\end{proof}

\section{Minimal stretch factors for non-orientable surfaces with marked points}
\label{sec:application}

In this section we will use \autoref{thm:classifying-fibrations} and Proposition  \ref{prop:asymptotic} to adapt the methods of Yazdi \cite{yazdi2018pseudo} to non-orientable surfaces. We recall the statement of the main theorem:
\begin{manualtheorem}
  {\ref{thm:stretch1}}
Let $\no_{g,n}$ be a non-orientable surface of genus $g$ with $n$ punctures, and let $\ell_{g,n}'$ be the logarithm of
  the minimum stretch factor of the pseudo-Anosov mapping classes acting on $\no_{g,n}$.
  Then for any fixed $n \in \mathbb{N}$, there is a positive constant $B'_1 = B'_1(n)$ and $B'_2 = B'_2(n)$ such
  that for any $g \geq 3$,
  the quantity $\ell_{g,n}'$ satisfies the following inequalities:
  \begin{align*}
    \frac{B'_1}{g} \leq \ell'_{g,n} \leq \frac{B'_2}{g}.
  \end{align*}
\end{manualtheorem}

Observe that the lower bound for the non-orientable case follows easily from the lower bound for the orientable case.
Indeed, let $\varphi$ be a pseudo-Anosov map with the minimal stretch factor on $\no_{g,n}$. The orientation double cover of $\no_{g,n}$ is $\os_{G,2n}$, where $G = g-1$.  Note that in the non-orientable case we measure genus as the number of copies of the projective plane attached to $S^2$ via a connect sum and in the orientable case we measure genus as the number of copies of the torus attached to $S^2$ via a connected sum. Let $\wt{\varphi}:\os_{G, 2n}\to\os_{G,2n}$ be the orientation preserving lift of $\varphi$.
By Proposition \ref{prop:2}, $\wt{\varphi}$ has the same
stretch factor as $\varphi$. The logarithm of the former is bounded below by $\frac{B_1}{G}$ (where $B_1$ is given by Yazdi \cite{yazdi2018pseudo}), and thus the stretch factor of $\varphi$ is bounded
below as well. The more challenging part of the proof is showing that the upper bound holds.

We will closely follow Yazdi's construction, which proceeds in five steps, though we will reorder them for clarity.  In steps 1 and 2, we construct a family of psuedo-Anosov homeomorphisms
of $\no_{g_i,n}$, where $\{g_i\}$ is an unbounded increasing sequence. However the sequence $\{g_i\}$ does not contain all natural numbers.  In step 3 we give an upper bound to the stretch factor of the previously constructed homeomorphisms. In steps 4 and 5, we construct pseudo-Anosov maps on surfaces of genera that do not belong to the sequence $\{g_i\}$. It is in steps 4 and 5 that we uses
Thurston's fibered face theory. We have adapted  each of Yazdi's five steps to work for non-orientable surfaces.

\subsection*{Step 1: Constructing the surfaces}

We begin by defining a family of surfaces $P_{n,k}$. Let $S$ be an orientable surface of genus 5 with 3
boundary components.  Call the boundary components $c,d$ and $e$. Choose an orientation of $S$ and let $c,d$ and $e$ inherit the induced orientations. Let $p$ and $q$ be marked points in the boundary component $e$. In Step 5 we will remove $p$ and all its copies.  Let  $r$ and
$s$ be the components of $e\setminus\{p,q\}$. We obtain a non-orientable surface $T$ from $S$ by adding two cross caps to $S$ (retaining the orientation of the boundary components of $S$).  The resulting surface $T$ is shown in \autoref{fig:buildingblock}.

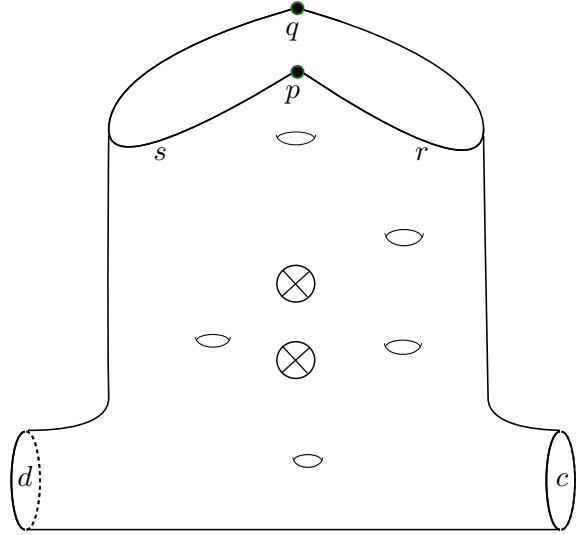
\begin{figure}[ht]
    \centering
    \def\svgscale{0.8}
\begingroup%
  \makeatletter%
  \providecommand\color[2][]{%
    \errmessage{(Inkscape) Color is used for the text in Inkscape, but the package 'color.sty' is not loaded}%
    \renewcommand\color[2][]{}%
  }%
  \providecommand\transparent[1]{%
    \errmessage{(Inkscape) Transparency is used (non-zero) for the text in Inkscape, but the package 'transparent.sty' is not loaded}%
    \renewcommand\transparent[1]{}%
  }%
  \providecommand\rotatebox[2]{#2}%
  \newcommand*\fsize{\dimexpr\f@size pt\relax}%
  \newcommand*\lineheight[1]{\fontsize{\fsize}{#1\fsize}\selectfont}%
  \ifx\svgwidth\undefined%
    \setlength{\unitlength}{266.33913055bp}%
    \ifx\svgscale\undefined%
      \relax%
    \else%
      \setlength{\unitlength}{\unitlength * \real{\svgscale}}%
    \fi%
  \else%
    \setlength{\unitlength}{\svgwidth}%
  \fi%
  \global\let\svgwidth\undefined%
  \global\let\svgscale\undefined%
  \makeatother%
  \begin{picture}(1,0.93669087)%
    \lineheight{1}%
    \setlength\tabcolsep{0pt}%
    \put(0,0){\includegraphics[width=\unitlength,page=1]{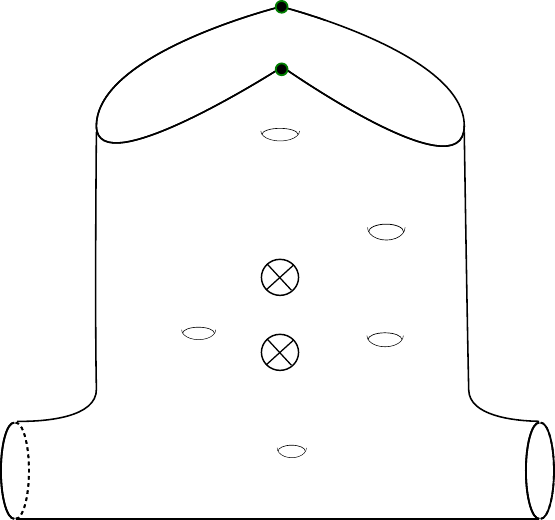}}%
    \put(0.25319527,0.65619769){\makebox(0,0)[lt]{\lineheight{1.25}\smash{\begin{tabular}[t]{l}$s$\end{tabular}}}}%
    \put(0.48630779,0.76528549){\makebox(0,0)[lt]{\lineheight{1.25}\smash{\begin{tabular}[t]{l}$p$\end{tabular}}}}%
    \put(0.71557712,0.65619769){\makebox(0,0)[lt]{\lineheight{1.25}\smash{\begin{tabular}[t]{l}$r$\end{tabular}}}}%
    \put(0.48630779,0.87748639){\makebox(0,0)[lt]{\lineheight{1.25}\smash{\begin{tabular}[t]{l}$q$\end{tabular}}}}%
    \put(0.01248441,0.07929169){\makebox(0,0)[lt]{\lineheight{1.25}\smash{\begin{tabular}[t]{l}$d$\end{tabular}}}}%
    \put(0.96451262, 0.07929169){\makebox(0,0)[lt]{\lineheight{1.25}\smash{\begin{tabular}[t]{l}$c$\end{tabular}}}}%
  \end{picture}%
\endgroup%

    \caption{The surface $T$, which will be the building block of the construction.}
    \label{fig:buildingblock}
\end{figure}

Let $T_{i,j}$ be copies of the surface $T$, where $i,j \in \mathbb{Z}$. Let $c_{i,j}, d_{i,j}$ and $e_{i,j}$ be the (oriented) boundary components of $T_{i,j}$ and let $r_{i,j}$ and $s_{i,j}$ be the copies of the arcs $r$ and $s$ in $T_{i,j}$. Define a connected infinite surface $T_\infty$ as the quotient:
\begin{align*}
  T_\infty \coloneqq \left. \left( \bigcup T_{i,j} \right)\right/\sim
\end{align*}
for all integers $i$ and $j$. The gluing $\sim$ is given by orientation-reversing identifications:
\begin{align}
  c_{i,j} &\sim d_{i+1,j} \label{identification1}\\
  r_{i,j} &\sim s_{i,j+1}.\label{identification2}
\end{align}

We have two
natural shift maps $\overline{\rho_1},\overline{\rho_2}: T_\infty \to T_\infty$ that act in the
following manner:
\begin{align*}
  \overline{\rho_1}: T_{i,j} &\mapsto T_{i+1, j} \\
  \overline{\rho_2}: T_{i,j} &\mapsto T_{i, j+1}.
\end{align*}

Note that $\overline{\rho_1}$ and $\overline{\rho_2}$ commute. Define the surface $P_{n,k}$ as the quotient of the surface $T_\infty$ by the
covering action of the group generated by $(\overline{\rho_1})^n$ and $(\overline{\rho_2})^k$. Then
$\overline{\rho_1}$ and $\overline{\rho_2}$ are equivariant with respect to the covering map.  We denote the induced homeomorphisms of the quotient $P_{n,k}$ by $\rho_1$
and $\rho_2$.  Note that later we will require that $k\geq 3$ and $n$ is the number of punctures, given in \autoref{thm:stretch1}.

\begin{lem}
\label{lem:genera}
Let
\begin{align*}
    g_{n,k} &= (14k - 2)n + 2
\end{align*} for $n \geq 1$ and $k \geq 1$.
    The genus of $P_{n,k}$ is $g_{n,k}$.
\end{lem}
\begin{proof}
  Let $U \subset P_{n,k}$ be the subsurface
  \begin{align*}
    U = \left. \left( \bigcup_{j =0}^{k-1} T_{0,j} \right)\right/\sim'
  \end{align*}
  where $\sim'$ is given by (\ref{identification1}) and (\ref{identification2}) and by identifying $r_{i,k-1}$ and $s_{i,0}$.
  Then $U$ is a compact, non-orientable surface of genus $12k$ with $2k$ boundary components.
  The surface $P_{n,k}$ consists of $n$ copies of $U$ identified along the $2k$ boundary components.  Therefore the Euler characteristic of $P_{n,k}$ is:
  \begin{align*}
    \chi(P_{n,k}) &= n \cdot \chi(U)\\
                  &=n\cdot(2-12k-2k)\\
                  &= -n(14k - 2).
  \end{align*}
  Since $P_{n,k}$ is a non-orientable surface with empty boundary, we have that:
  \begin{equation*}
    g_{n,k} = n(14k-2) + 2. \qedhere
  \end{equation*}
\end{proof}

\subsection*{Step 2: Constructing the maps}
In what is now a classical paper, Penner gives a construction of pseudo-Anosov homeomorphisms on both orientable and non-orientable surfaces \cite{penner1988construction}.  Below we outline the Penner construction for non-orientable surfaces following the details of Liechti--Strenner \cite[Section 2]{LS}.

 \paragraph{Inconsistent markings} Let $\no$ be a non-orientable surface and let $c$ be a two-sided curve in $\no$.  There exists a neighborhood of $c$ that is homeomorphic to an annulus.  Let $\mathcal{A}_c$ be an annulus and let $\zeta_c: \mathcal{A}_c \xrightarrow{} \no$ be the homeomorphism that maps to a neighborhood of $c$.  The homeomorphism $\zeta_c$ is called a \textit{marking} of $c$. A pair consisting of a curve $c$ and $\zeta_c$ is called a {\it marked curve}.
 If we fix an
orientation of $\mathcal{A}_c$, then we can pushforward this orientation to $\no$. Let
$(c,\zeta_c)$ and $(d,\zeta_d)$ be two marked curves that intersect at one point $p$.  We say that $(c,\zeta_c)$ and $(d,\zeta_d)$ are {\it marked inconsistently} if the
pushforward of the orientation of $\mathcal{A}_c$ disagrees with the pushforward of the orientation of $\mathcal{A}_d$ in a neighborhood of $p$.  We emphasize that we can also say that two disjoint curves are inconsistently marked.

\paragraph{Dehn twists} We define the Dehn twist $\phi_{c,\zeta_c}(x)$ around a marked curve $(c,\zeta_c)$ as:
\begin{align*}
  \phi_{c,\zeta_c}(x) =
  \begin{cases}
    \zeta_c \circ \tau_c \circ \zeta_c^{-1}(x) & \text{for } x \in \zeta_c(\mathcal{A}_c) \\
    x & \text{for } x \in \no - \zeta_c(\mathcal{A}_c)
  \end{cases}.
\end{align*}
Here $\tau_c$ is the left-handed Dehn twist on $\mathcal{A}_c$, ie $\tau_c(\theta,t) = (\theta + 2\pi t,t)$.

\paragraph{The Penner construction for non-orientable surfaces} Let $\mathcal{C}$ be a set of marked essential simple closed curves in $\no$ such that no two curves in $\mathcal{C}$ are homotopic.  A Penner construction on $\no$ is a composition of Dehn twists about the marked curves in $\mathcal{C}$ such that:
\begin{enumerate}
\item the complement of curves in $\mathcal{C}$ in $\no$ consists of disks with at most one puncture or marked point,
    \item the marked curves $(c_i,\zeta_i),(c_j,\zeta_j)\in\mathcal{C}$ with $i\neq j$ are marked inconsistently,
    \item a Dehn twist about each marked curve in $\mathcal{C}$ is included in the composition, and
    \item all powers of Dehn twists are positive (alternatively, all powers are negative).
\end{enumerate}

\paragraph{Construction of $f_{n,k}$} We now construct homeomorphisms $f_{n,k}: P_{n,k} \to P_{n,k}$ that are defined as a composition of specific Dehn twists
followed by a finite order mapping class. The key insight is that a power of this map will be a composition of
Dehn twists that satisfy the criteria to be a Penner construction.  Therefore $f_{n,k}$ is pseudo-Anosov. Here we are using the rotational symmetry of the $P_{n,k}$.

\begin{figure}[t]
    \centering
    \def\svgscale{0.8}
    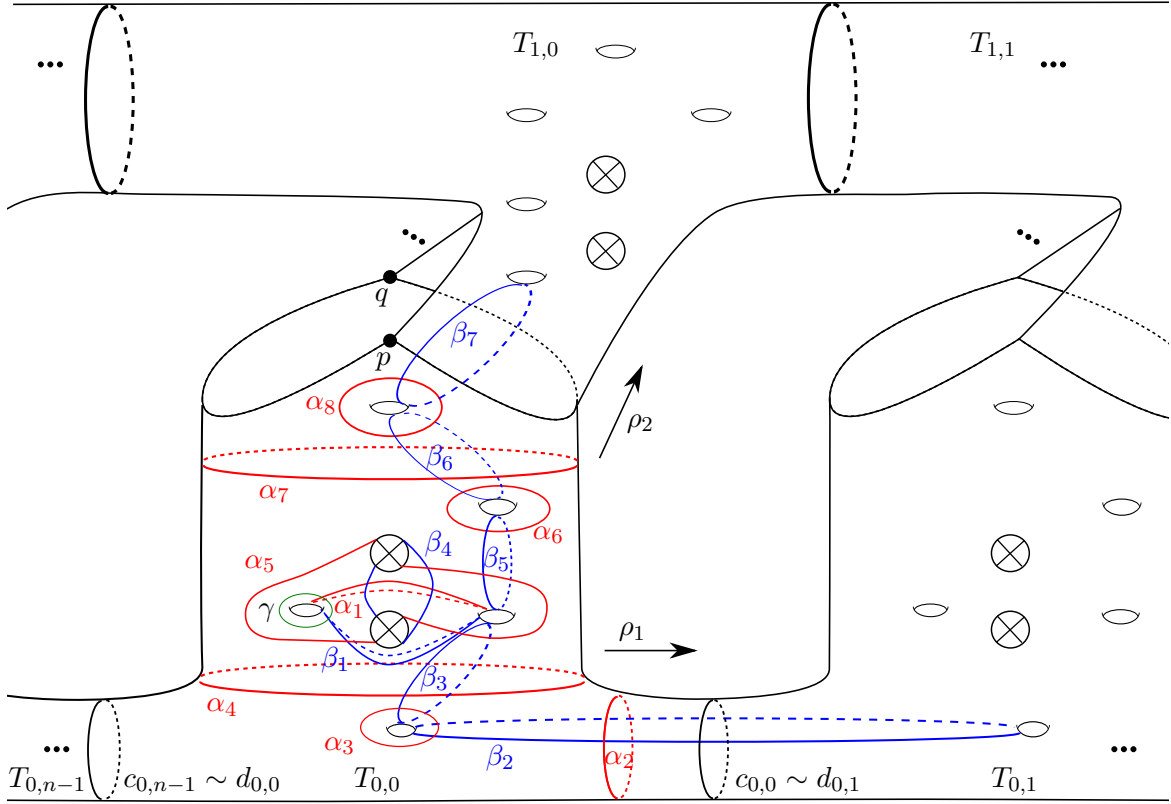

    \caption{Part of surface $P_{n,k}$ that includes the subsurface $T_{0,0}$ and the curves $\alpha_i$, $\beta_j$, and $\gamma$.}
    \label{fig:curves}
\end{figure}

Let $\{\alpha_1,\cdots,\alpha_8\}$ be the multi-curve in $T_{0,0}$ as shown in \autoref{fig:curves}.  Let $\{\beta_1,\cdots,\beta_7\}$ be the multi-curve in $T_{0,0}\cup T_{0,1}\cup T_{1,0}$ shown in \autoref{fig:curves}.

For any $\alpha_i$, we choose a marking $\zeta_{\alpha_i}$ to be orientation-preserving.  For any $\beta_j$ let $\zeta_{\beta_j}$ be
orientation reversing. From here forward, we will think of $\alpha_i$ and $\beta_j$ as (inconsistently) marked curves but we will suppress the marking maps. These choices give an inconsistent marking of $\{\alpha_1,\cdots,\alpha_8\}\cup\{\beta_1,\cdots,\beta_7\}$.

Let
$$\mathcal{R}=\displaystyle\bigcup_{i=2}^8\alpha_i.$$ Then $\mathcal{R}$ is a marked multi-curve that is disjoint from $\gamma$.  Let
$$\overline{\mathcal{R}}= \mathcal{R} \cup \rho_1(\mathcal{R}) \cup \dots \cup
\rho_1^{n-1}(\mathcal{R}).$$
Let $\Phi_r$ be the composition of Dehn twists about the marked curves in $\overline{\mathcal{R}}$.  Because the curves in $\overline{\mathcal{R}}$ are disjoint, the Dehn twists about the curves commute.

Let $$\mathcal{B}=\displaystyle\bigcup_{j=2}^7\beta_j$$ in $T_{0,0} \cup T_{0,1} \cup T_{1,0}$. As above, $\mathcal{B}$ is a marked multi-curve that is disjoint from $\gamma$.  Let $$\overline{\mathcal{B}} = \mathcal{B} \cup \rho_1(\mathcal{B}) \cup \dots \cup
\rho_1^{n-1}(\mathcal{B}).$$
Let $\Phi_b$
as the composition of Dehn twists about all of the marked curves in $\overline{\mathcal{B}}$.
As with $\overline{\mathcal{R}}$, the Dehn twists about curves in $\overline{\mathcal{B}}$
commute.

Let $\alpha_1,\beta_1 \subset T_{0,0}$ be the (marked) curves in \autoref{fig:curves}. Let $\Phi$ be the composition
of Dehn twists along all the curves $\alpha_1, \rho_1(\alpha_1), \dots, \rho_1^{n-1}(\alpha_1)$ followed by
Dehn twists along all the curves $\beta_1,\rho_1(\beta_1),\dots,\rho_1^{n-1}(\beta_1)$. Define the map $f_{n,k}$
as:
\begin{align*}
    f_{n,k} &\coloneqq \rho_2 \circ \Phi \circ \Phi_b \circ \Phi_r
\end{align*}
Since the curves about which we twist to construct $f_{n,k}$ satisfy the conditions of Penner's construction, $f_{n,k}$ is a pseudo-Anosov homeomorphism.

\subsection*{Step 3: Bounding the Stretch Factor}
Following Yazdi, our next goal is to find an upper bound for the stretch factor of the pseudo-Anosov homeomorphisms $f_{n,k}$.

\paragraph{Train tracks} Let $S$ be a surface.  A \textit{train track} in $S$ is graph embedded in $S$ with that property that for every vertex $v$ of valence three or greater, all edges adjacent to $v$ have the same tangent vector at $v$. Let $\varphi:S\rightarrow S$ be a pseudo-Anosov homeomorphism.  The map $\varphi$ is equipped with a train track whose image under $\varphi$ is homotopic to itself.  Such a train track is an \textit{invariant train track} associated to $\varphi$. Invariant train tracks have an associated matrix whose Perron-Frobenius eigenvalue is the stretch factor of $\varphi$.

Yazdi uses the Lemma \ref{lem:spectral} to bound the spectral radius of the associated matrices.

\begin{lem}[Lemma 2.3 of \cite{yazdi2018pseudo}]
\label{lem:spectral}
Let $A$ be a non-negative integral matrix, $\Gamma$ be the adjacency graph of $A$, and $V(\Gamma)$ the set of
vertices of $\Gamma$. For each $v \in V(\Gamma)$, define $v^+$ to be the set of vertices $u\in V(\Gamma)$ such that there
is an oriented edge from $v$ to $u$. Let $D$ and $k$ be fixed natural numbers. Assume the following conditions
hold for $\Gamma$:
\begin{enumerate}[(i)]
\item For each $v \in V(\Gamma)$ we have $\deg_{\text{out}}(v) \leq D$,
\item There is a partition $V(\Gamma) = V_1 \cup \dots \cup V_\ell$ such that for each $v \in V_i$ we have
  $v^+ \subset V_{i+1}$, for any $1 \leq i \leq \ell$ except possibly when $i = 1$ or 3 (indices are mod $\ell$),
\item For each $v \in V_1$, we have $v^+ \subset V_2 \cup V_3$,
\item For each $v \in V_3$ we have $v^+ \subset V_3 \cup V_4$, and for $u \in v^+ \cap V_3$ we have
  $u^+ \subset V_4$, and
\item For all $3 < j \leq k$ and each $v \in V_j$, the set $v^+$ consists of a single element.
\end{enumerate}

Then the spectral radius of $A^{\ell-1}$ is at most $4D^4$.

\end{lem}
\noindent With this result in hand, we can find an upper bound for the stretch factor of $f_{n,k}$

\begin{lem}\label{lem:upperbound}
  Let $\lambda_{n,k}$ be the stretch factor of $f_{n,k}$. Then there exists a universal positive constant $C'$ such that for every $n \geq 1$ and $k \geq 3$, we have the following upper bound:
  \begin{align*}
   \log(\lambda_{n,k}) \leq C'\frac{n}{g_{n,k}}.
  \end{align*}
\end{lem}

\begin{proof}
 We deliberately constructed our curves so that all intersections of the multi-curve $\{\alpha_1,\dots,\alpha_8\}$ and $\{\beta_1,\dots,\beta_7\}$ occur in the subsurface $T_{0,0}$. The curve $\beta_3$ intersects $\rho_2(\alpha_3)$ at one point in $T_{0,1}$ and $\beta_7$ intersects $\rho_1(\alpha_8)$ at one point in $T_{1,0}$.

  We define the following unions of marked curves:
\begin{align*}
  \mathcal{A} &\coloneqq \mathcal{B} \cup \mathcal{R} \cup \{\alpha_1,\beta_1\} =\bigcup_{i=1}^8\alpha_i\cup\bigcup_{j=1}^7\beta_j,\\
  \smallskip
  \overline{\mathcal{A}} &\coloneqq \mathcal{A} \cup \rho_1(\mathcal{A}) \cup \dots \cup \rho_1^{n-1}(\mathcal{A}),\text{ and}\\
  \widehat{\mathcal{A}} &\coloneqq \overline{\mathcal{A}} \cup \rho_2\left(\overline{\mathcal{A}}\right) \cup \dots \cup \rho_2^{k-1}\left(\overline{\mathcal{A}}\right).
\end{align*}

Because $f_{n,k}$ is pseudo-Anosov, it has a corresponding invariant train track $\tau$.
Let $V_{\tau}$ be the space of all measured foliations that can be obtained by varying the weights on the tracks of $\tau$.
This forms a finite dimensional cone of measures, all of which can be carried by the combinatorial train track $\tau$.
Furthermore, $f_{n,k}$ acts linearly on this cone, and leaves the cone invariant, since $\tau$ is an invariant track for $f_{n,k}$.
Consider now the transverse measure $\mu_{\delta}$ for any curve $\delta$ in $\widehat{\mathcal{A}}$.
This transverse measure is carried by $\tau$, and thus $\mu_{\delta}$ belongs in the cone of measures $V_{\tau}$.
Let $H$ be the subspace spanned by $\{\mu_\delta \mid \delta\subset\widehat{\mathcal{A}}\}$.
This linear subspace is also left invariant by $f_{n,k}$.  Let $M$
be the matrix representing the linear action of $f_{n,k}$ on $H$ with respect to the basis $\{\mu_\delta \mid \delta\subset\widehat{\mathcal{A}}\}$.  Let $\Gamma$ be the adjacency graph for $M$. Work of Penner \cite{penner1988construction} tells us that the Perron--Frobenius eigenvalue of $M$ is the stretch factor of $f_{n,k}$.

To bound the spectral radius of $M$, we need to show that $\Gamma$ satisfies
the criteria of Lemma \ref{lem:spectral}.

\begin{enumerate}[(i)]
\item There exists a constant $D'$, independent of $n$ and $k$, such that for every curve $\delta \in \widehat{\mathcal{A}}$, the geometric
  intersection number between $\delta$ and every curve in $\overline{\mathcal{A}}$ is at most $D'$.
  Recall that $f_{n,k}=\rho_2\circ\Phi\circ\Phi_b\circ\Phi_r$.  Let $M_1,M_2,M_3$ and $M_4$ be the matrices describing the linear action of $\Phi_r,\,\Phi_b,\,\Phi$ and $\rho_2$ on $H$, respectively. The matrix
  $M$ can then be written as a product:
  \begin{align*}
    M = M_4M_3M_2M_1.
  \end{align*}
  For a curve $\delta \in \widehat{\mathcal{A}}$, the $L^1$-norm of $M_i(\mu_\delta)$ is bounded above by the geometric intersection of
  $f_{n,k}(\delta)$ with the curves in $\overline{\mathcal{A}}$.  Thus each of $M_1$, $M_2$, and $M_3$ will change the norm by
  a factor of at most $(1 + D')$. Since $\rho_2$ will not change intersection numbers, $M_4$ will preserve the
  $L^1$-norm. If we let $D = (1 + D')^3$, then the outward degree of each vertex in $\Gamma$ is at most $D$.

\medskip
For the remaining conditions, we partition the vertices of $\Gamma$.  Observe
$$\widehat{\mathcal{A}} = \rho_{2}^{-1}(\overline{\mathcal{A}})\cup\overline{\mathcal{A}}\cup\bigcup_{i=3}^k \rho_2^{i-2}(\overline{\mathcal{A}}).$$ Then define $V_1$ as the vertices of $\Gamma$ corresponding to $\rho_2^{-1}(\overline{\mathcal{A}})$, the set $V_2$ as the vertices of $\Gamma$ corresponding to $\overline{\mathcal{A}}$, and $V_i$ for
$3 \leq i \leq k$ as the vertices of $\Gamma$ corresponding to elements in
$\rho_2^{i-2}(\overline{\mathcal{A}})$.
\item Suppose that $v \in V_i$ for $i \neq 1,3$, is a vertex
  that corresponds to $\mu_\delta$ for a curve $\delta \in \widehat{\mathcal{A}}$.  
  Then $\delta$ is disjoint from all curves in $\overline{\mathcal{A}}$.  The action of $\Phi\circ\Phi_b\circ\Phi_r$ on $\widehat{\mathcal{A}}$ will preserve the set $\rho_2^{(i-2)\mod k}(\overline{\mathcal{A}})$ for each $i\neq 1,3$.  In particular, $\{\mu_\delta \mid \delta\subset\widehat{\mathcal{A}}\}$ will also be in $\rho_2^{(i-2)\mod k}(\overline{\mathcal{A}})$.  
  Then $\rho_2$ will rotate the curve $\Phi\circ \Phi_b\circ\Phi_r(\delta)$ into the set $\rho_2^{(i-1)\mod k}(\overline{\mathcal{A}})$. That is: $f_{n,k}=\rho_2\circ\Phi\circ\Phi_b\circ\Phi_r$ maps $\mu_\delta\in H$ to $$\sum_{\zeta\in \mathcal{Z}}\mu_\zeta$$ where $\mathcal{Z}$ is a subset of $\rho_2^{(i-1)\mod k}(\overline{\mathcal{A}})$.  Therefore $f_{n,k}$ maps $v$ to a subset of $V_{i+1}$.
\item To verify the third condition, we first look at the vertices $v \in V_1$ such that $v^+ \not\subset V_2$. Such vertices will correspond to the curves in $\rho_2^{-1}(\overline{\mathcal{A}})$ that $\Phi\circ\Phi_b\circ\Phi_r$ maps to curves that are not in $\rho_2(\overline{\mathcal{A}})$.  Because $\rho_1$ and $\rho_2$ commute, we can write the curves of $\rho_2^{-1}(\overline{\mathcal{A}})$ as:
    $$\rho_2^{-1}(\overline{\mathcal{A}})=\rho_2^{-1}(\mathcal{A})\cup\rho_1(\rho_2^{-1}(\mathcal{A}))\cup\cdots\cup\rho_1^{n-1}(\rho_2^{-1}(\mathcal{A})).$$  The elements of $v^+$ that are not in $V_2$ correspond to the images of curves in $\rho_2^{-1}(\overline{\mathcal{A}})$ under $f_{n,k}$ that are not in $\overline{\mathcal{A}}$.
As in Yazdi, the only curves in $\rho_2^{-1}(\overline{\mathcal{A}})$ that intersect curves in $\overline{\mathcal{A}}$ are those in the set:
\begin{align*}
    \mathcal{X} = \{ \rho_1^i(\rho_2^{-1}(\beta_7))\,\mid\,0\leq i\leq n-1\}.
  \end{align*}
  Therefore $\Phi\circ\Phi_b\circ\Phi_r$ maps curves in $\mathcal{X}$ to curves in $\rho_2^{-1}(\overline{\mathcal{A}})\cup\overline{\mathcal{A}}$.  Then $f_{n,k}=\rho_2\circ\Phi\circ\Phi_b\circ\Phi_r$ maps curves in $\mathcal{X}$ to curves in $\overline{\mathcal{A}}\cup\rho_2(\overline{\mathcal{A}}).$
For any curve in $\mathcal{X}$, the corresponding vertex $v\in V_1$ will have
  $v^+ \subset V_2 \cup V_3.$  Moreover, $f_{n,k}$ maps the curves $\rho_2^{-1}(\overline{\mathcal{A}})\setminus \mathcal{X}$ to curves in $\overline{\mathcal{A}}$.  Thus for any vertex $v\in V_1$ that does not correspond to an element of $\mathcal{X}$, the set $v^+$ is contained in $V_2$.
\item  Similarly, we look for the $v \in V_3$ such that $v^+ \not\subset V_4$.  Such vertices will correspond to the curves in $\rho_2(\overline{\mathcal{A}})$ that $\Phi\circ\Phi_b\circ\Phi_r$ maps to curves that are not in $\rho_2^2(\overline{\mathcal{A}})$.  As above, we have: $$\rho_2(\overline{\mathcal{A}})=\rho_2(\mathcal{A})\cup\rho_1(\rho_2(\mathcal{A}))\cup\cdots\cup\rho_1^{n-1}(\rho_2(\mathcal{A})).$$
The elements of $v^+$ that are not in $V_4$ correspond to the the images of $\rho_2(\overline{\mathcal{A}})$ that intersect the curves in $\overline{\mathcal{A}}$.
The only vertices of $V_4$ that correspond to such curves are those in the set:
 $$\mathcal{Y}=\{\rho_1^i(\rho_2(\alpha_8))\,\mid\,0\leq i\leq n-1\}.$$ 

  For any element $v \in V_3$ corresponding to a curve in $\mathcal{Y}$ and any
  $u \in v^+ \cap V_3$, the vertex $u$ does not correspond to an element of $\mathcal{Y}$.  Therefore $u^+ \subset V_4$.
\item All the curves corresponding to an element of $V_j$, $3 < j \leq k$ are disjoint from all the curves in
  $\overline{\mathcal{A}}$. Thus, $f_{n,k}$ just acts by rotation.
\end{enumerate}

Let $\lambda = \lambda_{n,k}$ be the stretch factor of $f_{n,k}$.  By Lemma \ref{lem:spectral}, we have:
\begin{gather*}
    \lambda^{k-1} = \rho(M)^{k-1} = \rho(M^{k-1}) \leq 4D^4.
\end{gather*}
Then the logarithm of $\lambda$ satisfies:
$$\log(\lambda^{k-1})=(k-1)\cdot \log(\lambda) \leq \log(4D^4).$$
Then for $k\geq 2$
    $$\frac{k}{2}\log(\lambda) \leq (k-1)\log(\lambda) \leq \log(4D^4).$$
On the other hand, we know $g_{n,k} = (14k - 2)n + 2 \leq 14kn$ by Lemma \ref{lem:genera}. Therefore
\begin{align*}
    \log(\lambda) \leq 2\log(4D^4)\cdot\frac{1}{k} \leq 2\log(4D^4)\cdot \frac{14n}{g_{n,k}}.
\end{align*}
Let $C' \coloneqq 28\log(4D^4)$ to complete the result.
\end{proof}

\subsection*{Step 4: The Mapping Torus}

We have now constructed an infinite family of non-orientable surfaces $P_{n,k}$ and pseudo-Anosov homeomorphisms $f_{n,k}:P_{n,k}\to P_{n,k}$, but this is not
enough. In Lemma \ref{lem:genera}, we show that $\{P_{n,k}\}$ does not include surfaces of infinitely many genera. We use the strategy of McMullen \cite{mcmullen2000polynomial} and our extension of the Thurston's
fibered face theory to fill in the gaps.

Next we follow the strategy of Leininger--Margalit \cite{leininger2013number} to find a surface embedded in the mapping torus of minimal genus.  In our situation, this means that we will construct an embedded surface homeomorphic to $\no_3$.

\begin{figure}[t]
    \centering
    \def\svgscale{0.75}
\begingroup%
  \makeatletter%
  \providecommand\color[2][]{%
    \errmessage{(Inkscape) Color is used for the text in Inkscape, but the package 'color.sty' is not loaded}%
    \renewcommand\color[2][]{}%
  }%
  \providecommand\transparent[1]{%
    \errmessage{(Inkscape) Transparency is used (non-zero) for the text in Inkscape, but the package 'transparent.sty' is not loaded}%
    \renewcommand\transparent[1]{}%
  }%
  \providecommand\rotatebox[2]{#2}%
  \newcommand*\fsize{\dimexpr\f@size pt\relax}%
  \newcommand*\lineheight[1]{\fontsize{\fsize}{#1\fsize}\selectfont}%
  \ifx\svgwidth\undefined%
    \setlength{\unitlength}{266.33913055bp}%
    \ifx\svgscale\undefined%
      \relax%
    \else%
      \setlength{\unitlength}{\unitlength * \real{\svgscale}}%
    \fi%
  \else%
    \setlength{\unitlength}{\svgwidth}%
  \fi%
  \global\let\svgwidth\undefined%
  \global\let\svgscale\undefined%
  \makeatother%
  \begin{picture}(1,0.93669087)%
    \lineheight{1}%
    \setlength\tabcolsep{0pt}%
    \put(0,0){\includegraphics[width=\unitlength,page=1]{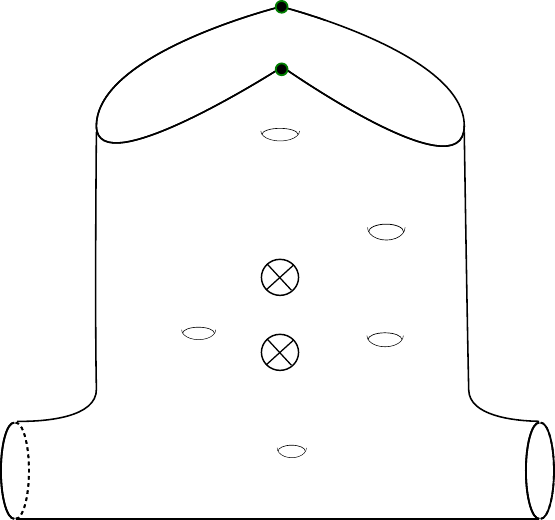}}%
    \put(0.25319527,0.65619769){\makebox(0,0)[lt]{\lineheight{1.25}\smash{\begin{tabular}[t]{l}$s$\end{tabular}}}}%
    \put(0.48630779,0.76528549){\makebox(0,0)[lt]{\lineheight{1.25}\smash{\begin{tabular}[t]{l}$p$\end{tabular}}}}%
    \put(0.71557712,0.65619769){\makebox(0,0)[lt]{\lineheight{1.25}\smash{\begin{tabular}[t]{l}$r$\end{tabular}}}}%
    \put(0.48630779,0.87748639){\makebox(0,0)[lt]{\lineheight{1.25}\smash{\begin{tabular}[t]{l}$q$\end{tabular}}}}%
    \put(0.01248441,0.07929169){\makebox(0,0)[lt]{\lineheight{1.25}\smash{\begin{tabular}[t]{l}$d$\end{tabular}}}}%
    \put(0.96451262, 0.07929169){\makebox(0,0)[lt]{\lineheight{1.25}\smash{\begin{tabular}[t]{l}$c$\end{tabular}}}}%
    \put(0,0){\includegraphics[width=\unitlength,page=2]{gamma_curves.pdf}}%
    \put(0.57809029,0.22834353){\makebox(0,0)[lt]{\lineheight{1.25}\smash{\begin{tabular}[t]{l}$\widehat{\gamma}$\end{tabular}}}}%
    \put(0.33967926,0.35720275){\makebox(0,0)[lt]{\lineheight{1.25}\smash{\begin{tabular}[t]{l}$\gamma$\end{tabular}}}}%
  \end{picture}%
\endgroup%

    \caption{The curves $\gamma$ and $\widehat{\gamma}$ bound an a non-orientable surface of genus 1.}
    \label{fig:gammacurves}
\end{figure}

\begin{prop}
\label{lem:genus3}
Let $M_{n,k}$ be the mapping torus of $f_{n,k}$. Let $\mathcal{K}_{n,k}$ denote the fibered cone of
$H^1(M_{n,k};\mathbb{R})$ corresponding to the map $f_{n,k}$.
There is a relatively orientable incompressible surface $F_{n,k}$ embedded in $M_{n,k}$ that is homeomorphic to $\mathcal{N}_3$.
Moreover $F_{n,k}$ is transverse to the suspension flow direction given by $f_{n,k}$ and the Poincar\'e dual of $F_{n,k}$ is in
the closure $\overline{\mathcal{K}_{n,k}}$.
\end{prop}
\begin{proof}
  Let $\gamma \subset T_{0,0}$ be the curve shown in \autoref{fig:gammacurves}. Note that $\gamma$ and $\Phi(\gamma)$ bound a non-orientable surface
  $\widehat{F}$ of genus 1 with boundary. For convenience, we will denote $\Phi(\gamma)$ by $\widehat{\gamma}$. We are going to follow the image of $\gamma$
  under powers of $f_{n,k}$.  Then we attach annuli to the
  boundary of $\widehat{F}$ to obtain $\mathcal{N}_3$. Since $\gamma$ is disjoint from all curves in $\overline{\mathcal{R}}$ and $\overline{\mathcal{B}}$ (as seen in \autoref{fig:curves}), the maps $\Phi_r$ and $\Phi_b$ act trivially on $\gamma$.  Recalling that $f_{n,k}=\rho_2\circ\Phi\circ\Phi_b\circ\Phi_r$, we have the following:
  \begin{align*}
    f_{n,k}(\gamma) &= \rho_2 \circ \Phi \circ \Phi_b \circ \Phi_r(\gamma) \\
                    &= \rho_2 \circ \Phi(\gamma) \\
                    &= \rho_2(\widehat{\gamma}).
  \end{align*}
  It follows that for all $1\leq i\leq k$, the curve $f_{n,k}^i(\gamma)$ is $\rho_2^i(\widehat{\gamma})$.
  For $1\leq i\leq k$, let $A_i$ be an annulus in $M_{n,k}$ that connects $f_{n,k}^{i-1}(\gamma)$ to $f_{n,k}^i(\gamma)$ obtained by following the suspension
  flow of $f_{n,k}$ around $M_{n,k}$. Let $A$ be the union of all of the $A_i$, which is also an annulus.  We can now construct the embedded surface $F_{n,k}$ by taking the union of
  $A$ and $\widehat{F}$. The union of $\widehat{F}$ with $A$ has empty boundary and Euler characteristic 0, so $F_{n,k}$ is homeomorphic to $\mathcal{N}_3$.

  We now need to show that $F_{n,k}$ is relatively orientable. We construct a outwards pointing normal vector field by combining the outwards pointing vector fields on $\widehat{F}$ and $A$ given by following $\gamma$ along the suspension flow.  Let $v_1$ be a vector field on $\widehat{F}$ pointing in the flow direction.  Define $v_2$ to be a vector field on $A$ as follows: on $\gamma$ define $v_2$ to be the vector field pointing in to $\widehat{F}$.  Flow the vector field along the suspension flow so $v_2$ is pointing away from $\widehat{F}$ on $\widehat{\gamma}$.

Let $U$ be a neighborhood of $\gamma$ in $F_{n,k}=\widehat{F}\cup A$. Define two bump functions $c_1$ and $c_2$ supported in $U$.  Let $c_1$ be $1$ on $\partial U\cap A$ and $0$ on $\widehat{F}$. Let $c_2$ be $1$ on $\partial U\cap \widehat{F}$ and $0$ on $A$.
  We add the vector fields $v_1$ on $\widehat{F}$ and $v_2$ on $A$ using these bump functions: the resulting vector field is $c_1 v_1 + c_2 v_2$.
  Observe that since $v_1$ points in the flow direction, and $v_2$ points into the surface, the new vector field $c_1 v_1 + c_2 v_2$ is transverse to $F_{n,k}$ in the neighborhood of $\gamma$ (see \autoref{fig:gluing-normal-field} for a picture of the resulting transverse vector field).

  We perform a similar construction in a small neighborhood of $\widehat{\gamma}$: in this case, the fact that the vector field on $\widehat{F}$ points in the flow direction, and the vector field on $A$ points away from the surface $\widehat{F}$ ensures that the new vector field is transverse, in a neighborhood of $\widehat{\gamma}$, to the surface $F_{n,k}$.

  A key fact we use in this construction is that the vector field along $A$ that starts pointing into $\widehat{F}$ at $\gamma$ comes back pointing away from the surface at $\widehat{\gamma}$: this is because the homeomorphism $f_{n,kj}$ maps the inner tubular neighborhood of $\gamma$ to the outer tubular neighborhood of $\widehat{\gamma}$, where inner and outer tubular neighborhood are the half tubular neighborhoods contained in $\widehat{F}$ and the complement respectively.
  This fact about $f_{n,k}$ follows from its definition, ie following the four homeomorphisms whose composition is $f_{n,k}$.
\begin{figure}[t]
  \centering
    \def\svgscale{0.5}
    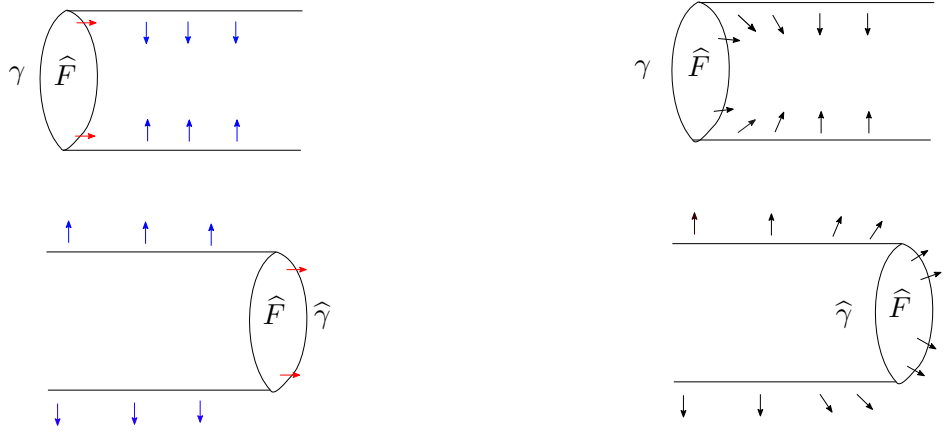

  \caption{The left side shows the vector fields $v_1$ on $\widehat{F}$ and $v_2$ on $A$.  The upper picture is a neighborhood of $\gamma$ and the lower picture is a neighborhood of $\widehat{\gamma}$.  The right side shows the vector fields $c_1v_1+c_2v_2$ on neighborhoods of $\gamma$ and $\widehat{\gamma}$.}
  \label{fig:gluing-normal-field}
\end{figure}

  The proof that $F_{n,k}$ can be isotoped to be transverse to the suspension flow is the same as the
  proof Yazdi uses \cite{yazdi2018pseudo}, which is a restatement of that of Leininger--Margalit \cite{leininger2013number}. We include it here for completeness.

  Let $N(\gamma)$ be a tubular neighborhood of $\gamma$ in $\widehat{F}$.  Let $\eta: \widehat{F} \xrightarrow{} [0,1]$ be a smooth function supported on $N(\gamma)$ with the following properties:
  \begin{itemize}
      \item $\eta^{-1}(1) = \gamma$ and
      \item the derivative of $\eta$ vanishes on $\gamma$.
    \end{itemize}
Let $\pi:M_{n,k}\rightarrow S^1$ be the projection map and let $t_0$ be such that $\widehat{F}\subset\pi^{-1}(t_0)$.  Let $g: \widehat{F} \xrightarrow{} M_{n,k}$ be the suspension flow of $f_{n,k}$ defined as $g(x) =(x,t_0+k\cdot\eta(x))$. Then the restriction of $g$ to the interior of $\widehat{F}$ is an embedding into $M_{n,k}$ and $g(\gamma) = \widehat{\gamma}$. Therefore the image of $\widehat{F}$ under $g$ is an embedded non-orientable surface of genus three. Moreover, $g(\widehat{F})$ is isotopic to the natural embedding of $F_{n,k}$ in $M_{n,k}$, and is transverse to the suspension flow.
  Therefore, the Poincar\'e dual of $F_{n,k}$ is in $\overline{\mathcal{K}_{n,k}}$ by \autoref{thm:classifying-fibrations}.

  Finally, $F_{n,k}$ is incompressible in $M_{n,k}$ because $M_{n,k}$ is hyperbolic, and $F_{n,k}$ is genus $3$, the
  lowest possible genus for a hyperbolic non-orientable surface.
\end{proof}

\subsection*{Step 5: Filling in the Gaps}
Recall that the family of surfaces $P_{n,k}$ that we have constructed have genera in the set $\{(14k-2)n+2\}$.
We now want to construct surfaces of genera not in the set $\{(14k-2)n+2\}$ and pseudo-Anosov homeomorphisms of those surfaces that have small stretch factors.  To do this we use the mapping torus $M_{n,k}$ of $P_{n,k}$ by $f_{n,k}:P_{n,k}\to P_{n,k}$. By Proposition \ref{lem:genus3}, there exists a relatively incompressible surface $F_{n,k}$ in $M_{n,k}$ that is homeomorphic to $\no_3$.  Let $P_{n,k}^r$ be the oriented sum of $P_{n,k}$ and
$rF_{n,k}$, as defined in Proposition $\ref{thm:oriented-sum}$.  The surfaces $P_{n,k}^r$ will be surfaces of the remaining genera.

\begin{lem}
  The surface $P^r_{n,k}$ is of genus $g^r_{n,k} = g_{n,k} + r$. In particular, as $r$ varies between
  $0$ and $14n$, the genera of $\{P^r_{n,k}\}$ spans the range between $g_{n,k}$ and $g_{n,k+1}$. Moreover,
  $P^r_{n,k}$ is isotopic to a fiber of a fibration of $M_{n,k}$ with pseudo-Anosov monodromy that fixes $2n$
  of the singularities of its invariant foliation.
\end{lem}

\begin{proof}
  The Euler characteristic of an oriented sum is the sum of the Euler characteristics of the summands:
  \begin{align*}
    \chi(P^r_{n,k}) &= \chi(P_{n,k}) + r\cdot\chi(F_{n,k}) \\
                    &= (-2g_{n,k} + 2)-2r \\
                    &= -2(g_{n,k} + r) + 2.
  \end{align*}
  Since $P_{n,k}^r$ has no boundary or punctures, we have that the genus of $P_{n,k}^r$ is $g_{n,k}+r$.

  By Lemma \ref{lem:genus3} we know that $F_{n,k}$ is incompressible and transverse to the suspension flow given by $f_{n,k}$.  Therefore by Proposition \ref{prop:asymptotic}, there is a pseudo-Anosov homeomorphism $f_{n,k}^r$ of $P_{n,k}^r=P_{n,k}+rF_{n,k}$.

 As in Yazdi \cite[Lemma 3.5]{yazdibounds}, $f_{n,k}$ fixes the $2n$ singularities of the stable foliation that are the intersection points of the axis of $\rho_1$ with
  $P_{n,k}$. By Lemma \ref{lem:genus3}, the surface $F_{n,k}$ can be isotoped to be transverse to
  the suspension flow and disjoint from the orbit of the $2n$ singularities of $f_{n,k}$.  Hence the monodromy
  $f^r_{n,k}$ still fixes the corresponding $2n$ singularities on $P^r_{n,k}$.
\end{proof}

We now prove the non-orientable version of the final piece of Yazdi's proof \cite[Lemma 3.6]{yazdi2018pseudo}.
\begin{lem}
\label{lem:bound}
Let $\lambda_{n,k}^r$ be the stretch factor of $f_{n,k}^r:P_{n,k}^r\rightarrow P_{n,k}^r$. Then there exists a constant $C > 0$ such that for every $n \geq 1$, $k \geq 3$, and $0 \leq r \leq 14n$ we have the following upper bound on $\log(\lambda_{n,k}^r)$:
\begin{align*}
  \log(\lambda^r_{n,k}) \leq C\frac{n}{g^r_{n,k}}.
\end{align*}
\end{lem}
\begin{proof}
  Let $\mathcal{K} = \mathcal{K}_{n,k}$ be the fibered cone in $H^1(M_{n,k};\mathbb{R})$ corresponding to $f_{n,k}$ and $h: \mathcal{K} \xrightarrow[]{} \mathbb{R}$
  the function described in \autoref{thm:fm}. Note that $g_{n,k}\geq 42$, therefore we have the following bounds on $g_{n,k}^r$:
  \begin{align*}
    g^r_{n,k} &= g_{n,k} + r \\
              &\leq g_{n,k} + 14n \\
              &< 2g_{n,k}.
  \end{align*}
  Let $\omega$ be the Poincar\'e dual of $P^r_{n,k}$ and $\alpha$ the Poincar\'e dual of $P_{n,k}$.  Then the following string of inequalities holds:
  \begin{align*}
    h([\omega]) &< h([\alpha]) &&\text{(Convexity of $h$)}\\
                &\leq C'\frac{n}{g_{n,k}} &&\text{(Lemma \ref{lem:upperbound})}\\
                &\leq 2C'\frac{n}{g^r_{n,k}}&&\text{(upper bound for $g_{n,k}^r$)}.\qedhere
  \end{align*}
\end{proof}

In the initial construction of $P_{n,k}$, there were $2n$  marked points, which were singularities of the map $f_{n,k}$.  By the construction of $P_{n,k}^r$, these marked points are also singularities of $f^r_{n,k}$.  Now we puncture $P_{n,k}^r$ at $n$ of these marked points.  We could think of this as removing all copies of the point $p$ in the construction of $P_{n,k}$ in step 1.

We can now give a proof of \autoref{thm:stretch1}.

\begin{proof}[Proof of \autoref{thm:stretch1}]
As above, the lower bound follows easily from the lower bound in the orientable setting.

  To find the upper bound, let $C'=\frac{C}{2}$ be the value given in Lemma \ref{lem:bound}. Let $B'_2(n)$ be the
  following quantity:
  \begin{align*}
    B'_2(n) = \max\{2C'n, \ell'_{1,n}, 2\ell'_{2,n}, \dots, (40n + 1)\ell'_{40n+1,n}\}.
  \end{align*}
  By Proposition \ref{prop:asymptotic} and Lemma \ref{lem:bound}, $B'_2(n)$ is an upper bound for $g\cdot \ell'_{g,n}$.
\end{proof}

\bibliographystyle{plain}
\bibliography{references}

\end{document}